\documentclass[letterpaper,10pt]{article}
\usepackage{fullpage}

\usepackage{graphicx}%
\usepackage{multirow}%
\usepackage{amsmath,amssymb,amsfonts}%
\usepackage{amsthm}%
\usepackage{mathrsfs}%
\usepackage[title]{appendix}%
\usepackage{xcolor}%
\usepackage{textcomp}%
\usepackage{manyfoot}%
\usepackage{booktabs}%
\usepackage{listings}%
\usepackage{caption}
\usepackage{hyperref}
\usepackage[capitalize,nameinlink,noabbrev]{cleveref}
\usepackage{mathtools}
\usepackage{cite}
\usepackage{algorithm}
\usepackage{algpseudocode}
\usepackage{dashrule}
\usepackage{booktabs}

\crefname{ALC@line}{line}{lines}
\Crefname{ALC@line}{Line}{Lines}

\usepackage{mathabx}

\usepackage{tikz}
\usetikzlibrary{arrows.meta,positioning,calc,decorations.pathreplacing}

\tikzset{
  box/.style   ={draw, thick, minimum width=10mm, minimum height=7mm, align=center},
  bigbox/.style={draw, thick, minimum width=18mm, minimum height=7mm, align=center},
  arr/.style   ={-{Stealth[length=2.2mm]}, thick},
}

\usepackage{comment}
\usepackage{siunitx}
\usepackage{relsize}
\usepackage{ifthen}
\usepackage[colorinlistoftodos]{todonotes}

\usepackage[caption=false]{subfig}

\usepackage{graphics} 
\usepackage{rotating}
\usepackage{color}
\usepackage{enumerate}
\usepackage[T1]{fontenc}
\usepackage{psfrag}
\usepackage{epsfig} 
\usepackage{booktabs}
\usepackage{graphicx,url}
\usepackage{multirow}
\usepackage{array}
\usepackage{latexsym}
\usepackage{amsfonts}
\usepackage{amsmath}
\usepackage{amssymb}
\usepackage{mathtools}
\usepackage{xstring}
\usepackage{multirow}
\usepackage{xcolor}
\usepackage{prettyref}
\usepackage{flexisym}
\usepackage{bigdelim}
\usepackage{breqn} 
\usepackage{listings}

\usepackage{enumitem}
\usepackage{xspace}
\usepackage{bm}
\graphicspath{{./figures/}}
\usepackage{tikz}
\usetikzlibrary{matrix,calc}

\usepackage{amsmath,array,arydshln,xparse}
\usepackage{adjustbox} 
\usepackage{nicematrix}

\usepackage{mdwlist}

\makecompactlist{itemize}{stditemize}

\newrefformat{prob}{Problem\,\ref{#1}}
\newrefformat{def}{Definition\,\ref{#1}}
\newrefformat{sec}{Section\,\ref{#1}}
\newrefformat{sub}{Section\,\ref{#1}}
\newrefformat{prop}{Proposition\,\ref{#1}}
\newrefformat{app}{Appendix\,\ref{#1}}
\newrefformat{alg}{Algorithm\,\ref{#1}}
\newrefformat{cor}{Corollary\,\ref{#1}}
\newrefformat{thm}{Theorem\,\ref{#1}}
\newrefformat{lem}{Lemma\,\ref{#1}}
\newrefformat{fig}{Fig.\,\ref{#1}}
\newrefformat{tab}{Table\,\ref{#1}}

\newtheorem{corollary}{Corollary}
\newtheorem{lemma}{Lemma}

\newcommand{\cf}{\emph{cf.}\xspace}

\newcommand{\bdmath}{\begin{dmath}}
\newcommand{\edmath}{\end{dmath}}
\newcommand{\beq}{\begin{equation}}
\newcommand{\eeq}{\end{equation}}
\newcommand{\bdm}{\begin{displaymath}}
\newcommand{\edm}{\end{displaymath}}
\newcommand{\bea}{\begin{eqnarray}}
\newcommand{\eea}{\end{eqnarray}}
\newcommand{\beal}{\beq \begin{array}{ll}}
\newcommand{\eeal}{\end{array} \eeq}
\newcommand{\beas}{\begin{eqnarray*}}
\newcommand{\eeas}{\end{eqnarray*}}
\newcommand{\ba}{\begin{array}}
\newcommand{\ea}{\end{array}}
\newcommand{\bit}{\begin{itemize}}
\newcommand{\eit}{\end{itemize}}
\newcommand{\ben}{\begin{enumerate}}
\newcommand{\een}{\end{enumerate}}

\newcommand{\calA}{{\cal A}}
\newcommand{\calB}{{\cal B}}
\newcommand{\calC}{{\cal C}}

\newcommand{\calE}{{\cal E}}
\newcommand{\calF}{{\cal F}}
\newcommand{\calG}{{\cal G}}
\newcommand{\calH}{{\cal H}}
\newcommand{\calI}{{\cal I}}
\newcommand{\calJ}{{\cal J}}
\newcommand{\calK}{{\cal K}}
\newcommand{\calL}{{\cal L}}
\newcommand{\calM}{{\cal M}}
\newcommand{\calN}{{\cal N}}
\newcommand{\calO}{{\cal O}}
\newcommand{\calP}{{\cal P}}

\newcommand{\calU}{{\cal U}}
\newcommand{\calV}{{\cal V}}
\newcommand{\calW}{{\cal W}}
\newcommand{\calX}{{\cal X}}

\newcommand{\ie}{\emph{i.e.,}\xspace}

\newcommand{\hide}[1]{}

\newcommand{\hiddenText}{{\color{gray} hidden text.}}
\newcommand{\hideWithText}[1]{\hiddenText}

\newcommand{\Natural}[1]{ { {\mathbb N}^{#1} } }

\DeclareMathOperator*{\argmin}{arg\,min}

\newcommand{\norm}[1]{\left\| #1 \right\|}

\newcommand{\tran}{^{\mathsf{T}}}

\newcommand{\diag}[1]{\mathrm{diag}\left(#1\right)}
\newcommand{\trace}[1]{\mathrm{tr}\left(#1\right)}
\newcommand{\rank}[1]{\mathrm{rank}(#1)}

\newcommand{\vect}[1]{\left[\begin{array}{c}  #1  \end{array}\right]}

\newcommand{\Real}[1]{ { {\mathbb R}^{#1} } }
\newcommand{\reals}{\Real{}}

\newcommand{\scenario}[1]{{\smaller \sf#1}\xspace}

\newcommand{\MOSEK}{\scenario{MOSEK}}

\newcommand{\blue}[1]{{\color{blue}#1}}

\newcommand{\linkToPdf}[1]{\href{#1}{\blue{(pdf)}}}
\newcommand{\linkToPpt}[1]{\href{#1}{\blue{(ppt)}}}
\newcommand{\linkToCode}[1]{\href{#1}{\blue{(code)}}}
\newcommand{\linkToWeb}[1]{\href{#1}{\blue{(web)}}}
\newcommand{\linkToVideo}[1]{\href{#1}{\blue{(video)}}}
\newcommand{\linkToMedia}[1]{\href{#1}{\blue{(media)}}}
\newcommand{\award}[1]{\xspace} 

\renewcommand{\norm}[1]{\lVert #1 \rVert}
\newcommand{\inprod}[2]{\left\langle #1, #2 \right\rangle}

\newcommand{\mymid}{\ \middle\vert\ }

\newcommand{\sym}[1]{\mathbb{S}^{#1}}

\newcommand{\bmat}{\left[ \begin{array}}
\newcommand{\emat}{\end{array}\right]}
\newcommand{\psd}[1]{\sym{#1}_{+}}
\newcommand{\pd}[1]{\sym{#1}_{++}}

\newcommand{\abs}[1]{\left|#1\right|}

\newcommand{\bbN}{\mathbb{N}}

\newcommand{\normtwo}[1]{\norm{#1}_2}

\newcommand{\normF}[1]{\norm{#1}_\mathsf{F}}

\DeclareDocumentCommand{\svec}{o}{
  \IfNoValueTF{#1}{\mathrm{svec}}{\svec \left( #1 \right)}
}
\DeclareDocumentCommand{\smat}{o}{
  \IfNoValueTF{#1}{\mathrm{smat}}{\smat \left( #1 \right)}
}
\newcommand{\HomPoly}[1]{H_{#1}}
\newcommand{\NonPolyCone}[1]{P_{#1}}
\newcommand{\SosCone}[1]{\Sigma_{#1}}
\newcommand{\MomConePs}[1]{\calM_{#1}^{\mathsf{ps}}}
\newcommand{\MomConeAtm}[1]{\calM_{#1}^{\mathsf{atm}}}
\newcommand{\Hankel}[1]{\calH_{#1}}
\newcommand{\MultiIndex}[2]{\bbN_{#2}^{#1}}
\newcommand{\MultiIndexLeq}[2]{\bbN_{\leq #2}^{#1}}
\newcommand{\stimes}{\vee}
\newcommand{\SymGroup}[1]{\mathfrak{S}_{#1}}
\newcommand{\TenProd}[2]{{#1}^{\otimes #2}}
\newcommand{\SymTenProd}[2]{\mathcal{S}^{#2}(#1)}
\newcommand{\SymProj}[2]{\calP_{#1, #2}^{\stimes}}
\newcommand{\cone}{\mathrm{cone}}
\newcommand{\linspan}{\mathrm{span}}
\newcommand{\aff}{\mathrm{aff}}
\newcommand{\conv}{\mathrm{conv}}
\newcommand{\SymRankOne}{\calX_{\mathrm{sep}}^{\stimes}}
\newcommand{\clE}{\mathrm{cl}_{\mathrm{E}}}
\DeclareDocumentCommand{\clZ}{o}{
  \IfNoValueTF{#1}{\mathrm{cl}_{\mathrm{Z}}}{\mathrm{cl}_{\mathrm{Z}}^{#1}}
}
\newcommand{\ri}{\mathrm{ri}}
\newcommand{\minface}{\mathrm{minface}}

\newcommand{\titlemath}[1]{\texorpdfstring{#1}{M}}
\renewcommand{\dim}{\mathrm{dim}}
\newcommand{\codim}{\mathrm{codim}}
\newcommand{\tz}{\tilde{z}}
\newcommand{\tX}{\widetilde{X}}
\newcommand{\tM}{\widetilde{M}}
\newcommand{\tcalA}{\widetilde{\calA}}
\newcommand{\tcalF}{\widetilde{\calF}}

\newcommand{\im}{\mathrm{im}}
\newcommand{\ray}{\reals_+}
\newcommand{\tr}{\mathrm{tr}}

\newcommand{\epsbreak}{\varepsilon_{\mathsf{break}}}
\newcommand{\epsrank}{\varepsilon_{\mathsf{rank}}}
\newcommand{\epscol}{\varepsilon_{\mathsf{col}}}
\newcommand{\epsalt}{\varepsilon_{\mathsf{alt}}}
\newcommand{\Talt}{T_{\mathsf{alt}}}

\newcommand{\Tipm}{T_{\mathsf{ipm}}}

\renewcommand{\vect}{\mathrm{vec}}

\hypersetup{colorlinks, hypertexnames=false, pageanchor=true,
            linkcolor=blue, citecolor={green!50!black}, urlcolor=cyan,
            pdftitle={Field Analysis in ADMM for SDP}
            }
\makeatletter
\AddToHook{cmd/appendix/before}{\def\cref@section@alias{appendix}}
\makeatother
\mathtoolsset{centercolon}
\crefname{assumption}{Assumption}{Assumptions}

\theoremstyle{thmstyleone}%
\newtheorem{theorem}{Theorem}%

\theoremstyle{thmstyletwo}%

\theoremstyle{thmstylethree}%

\title{
    Simplicial Regularizability of the Pseudo-Moment Cone and Carath\'eodory-Type Atomic Decomposition of Moment Matrices
}
\author{Shucheng Kang%
\thanks{School of Engineering and Applied Sciences, Harvard University. Email: \texttt{skang1@g.harvard.edu}}
\and Heng Yang%
\thanks{School of Engineering and Applied Sciences, Harvard University. Email: \texttt{hankyang@seas.harvard.edu}}
}

\pdfoutput=1
\begin{document}

\maketitle


\begin{abstract}
    We study the facial geometry of the (homogeneous) pseudo-moment cone \( \Sigma_{n,2d}^* \) and its implications for atomic decomposition of moment matrices. For fixed \( d \ge 2 \), we show that if a moment matrix is formed by \( O(n^d) \) generically chosen weighted atoms, then its minimal face in the matrix realization of the pseudo-moment cone is \emph{simplicial} and generated by the planted rank-one atoms. 
    Based on this geometric result, we develop a Carath\'eodory-type extreme-ray decomposition algorithm for spectrahedral cones and show that, when specialized to the pseudo-moment cone, it yields an efficient atomic decomposition method for generically generated moment matrices in the same regime. A stabilized numerical implementation demonstrates strong recovery performance, and suggests that outside the guaranteed regime the algorithm may serve as a practical sampler of high-rank extreme rays.
\end{abstract}






\section{Introduction}
\label{sec:intro}


The (homogeneous) pseudo-moment cone is a fundamental geometric object in polynomial optimization, moment theory, and convex algebraic geometry. For instance, moment relaxations of polynomial optimization problems may be formulated as linear programs over this cone~\cite{lasserre01siopt-global,parrilo03mp-semidefinite}, while atomic decomposition of a moment matrix may be interpreted as a conic decomposition into rank-one point-evaluation rays~\cite{blekherman12jams-nonnegative-polynomial-sos,henrion25arxiv-positively-not-sos}. Following~\cite{blekherman12jams-nonnegative-polynomial-sos}, we denote this cone by \( \SosCone{n,2d}^* \), where \( n \) is the number of variables and \( 2d \) is the degree.

Despite its central role, our understanding of the facial structure of \( \SosCone{n,2d}^* \) remains limited. A major obstacle is the existence of higher-rank extreme rays---rays that do not correspond to point evaluations~\cite{blekherman15pams-positive-gorenstein-ideals}. Existing results mainly focus on deciding the ranks of such rays and their construction~\cite{blekherman12jams-nonnegative-polynomial-sos,blekherman15pams-positive-gorenstein-ideals,blekherman17jsc-extreme,henrion25arxiv-positively-not-sos}. In particular, there exist extreme rays of \( \SosCone{n,2d}^* \) of rank \( 3d-2 \) for \( d \ge 3 \) and \( n \ge 3 \), and of rank \( 6 \) for \( d = 2 \) and \( n \ge 4 \)~\cite{blekherman15pams-positive-gorenstein-ideals}. The ternary case has also been studied in greater detail~\cite{blekherman17jsc-extreme}. While these works reveal the complexity of the global facial structure of \( \SosCone{n,2d}^* \), it is natural to ask whether simpler local geometric structures exist that can be leveraged for algorithm design.


\paragraph{Contributions.} We give an affirmative answer to the above question. Our contributions are threefold.
\begin{itemize}
    \item \textbf{Simplicial regularizability of the pseudo-moment cone.} Our main contribution is a new facial-geometric characterization of the pseudo-moment cone. For fixed \( d \ge 2 \), we show that if a moment matrix is formed by \( O(n^d) \) generically chosen weighted atoms, then its minimal face in the matrix realization of the pseudo-moment cone is simplicial and generated by the planted rank-one atoms (\cf Theorem~\ref{thm:sim:minface-condition}). We call this property \emph{simplicial regularizability}. Our proof is motivated by a geometric result on computing linear sections of varieties~\cite[Theorem 15]{johnston23focs-computing-linear-sections-varieties}.

    \item \textbf{Carath\'eodory-type atomic decomposition of moment matrices.} Based on this facial-geometric result, we develop a Carath\'eodory-type extreme-ray decomposition procedure for spectrahedral cones; see Algorithm~\ref{alg:intro:extremeray-decomp}. When specialized to \( \MomConePs{n,2d} \), the matrix realization of \( \SosCone{n,2d}^* \), it yields an efficient atomic decomposition method for generically generated moment matrices with up to \( \Theta(n^d) \) atoms. This leads to an asymptotic improvement over methods based on flatness extension, whose recovery upper bound is \( \Theta(n^{d-1}) \)~\cite{henrion05ppc-detecting,klep18siopt-minimizer-extraction-robust}, provided that one does not solve a hierarchy of semidefinite relaxations to generate high-order moments~\cite{nie14fcm-truncated-moment}. Conceptually, our algorithm can be viewed as a practical variant of the procedure underlying the Carath\'eodory-number bound in~\cite[Theorem 4]{ito17laa-bound-caratheodory-number}.

    \item \textbf{Stabilized numerical implementation.} We develop a stabilized numerical implementation of Algorithm~\ref{alg:intro:extremeray-decomp} to improve its practical robustness. Experiments demonstrate strong recovery performance, suggest that the proven recovery regime is conservative, and indicate that outside the guaranteed regime the algorithm may also serve as a practical sampler of high-rank extreme rays of \( \SosCone{n,2d}^* \).
\end{itemize}

\begin{algorithm}
    \caption{Carath\'eodory-Type Extreme Ray Decomposition \label{alg:intro:extremeray-decomp}}
    \begin{algorithmic}[1]
        \Require nonempty spectrahedral cone $\calC \subseteq \sym{N}$, initial point $X \ne 0 \in \calC$
        \Ensure $\{t_k, M_k\}_{k=1}^r$, with $X = \sum_{k=1}^r t_k M_k$, $t_k > 0$ and $\ray M_k$ an extreme ray of $\calC$ for all $1 \le k \le r$
        \Procedure{RayDecomp}{$\calC, X$}
            \State $k := 1$, $X_1 := X$
            \While{$X_k \ne 0$}
                \State $\calF_k := \minface(X_k, \calC)$ \label{alg:intro:line:Fk}
                \State $B_k :=$ random matrix from any absolutely continuous distribution on $\sym{N}$ \label{alg:intro:line:Bk}
                \State Generate $M_k \in \sym{N}$ from the following linear SDP:
                \begin{align}
                    \label{eq:intro:Mk}
                    M_k \in \argmin_{Y \in \calF_k, \trace{Y} = 1} \inprod{B_k}{Y}
                \end{align}
                \State Find 
                \begin{align}
                    \label{eq:intro:tk}
                    t_k := \sup \{ t \ge 0 \mid X_k - t M_k \in \calF_k \}
                \end{align}
                \State $X_{k+1} := X_k - t_k M_k$ 
                \State $k := k+1$
            \EndWhile
            \State \Return $r := k-1$, $\{t_k, M_k\}_{k=1}^r$
        \EndProcedure
    \end{algorithmic}
\end{algorithm}

\paragraph{Outline.} The rest of the paper is organized as follows. \S\ref{sec:related} reviews related work, and \S\ref{sec:pre} introduces the necessary notation and preliminaries. \S\ref{sec:sim} establishes the simplicial regularizability of the pseudo-moment cone. \S\ref{sec:ray} presents Algorithm~\ref{alg:intro:extremeray-decomp} together with its stabilized implementation. Numerical experiments are given in \S\ref{sec:exp}, and \S\ref{sec:conclusion} concludes with a discussion of future directions.


\section{Related Work}
\label{sec:related}

\paragraph{Facial structures of the moment cone and the pseudo-moment cone.} 
Prior work on pseudo-moment cones has focused mainly on extreme rays and related algebraic boundaries: the first nontrivial cases \( (n,2d)=(3,6) \) and \( (4,4) \) were characterized via Cayley--Bacharach relations~\cite{blekherman11cm-algebraic-boundary-soscone,blekherman12jams-nonnegative-polynomial-sos,blekherman12book-semidefinite}; positive Gorenstein ideals later provided a general framework and sharp rank thresholds for non-evaluation extreme rays~\cite{blekherman15pams-positive-gorenstein-ideals}; and, for ternary forms, Hankel spectrahedra and their algebraic boundaries were studied systematically~\cite{blekherman17jsc-extreme}. Exact-arithmetic constructions of pseudo-moment certificates and extreme rays were proposed more recently in~\cite{henrion25arxiv-positively-not-sos}. On the moment-cone side, truncated moment cones were studied through atom sets, determinacy, core varieties, exposed faces, facial dimensions, and Carath\'eodory numbers~\cite{di18jfa-truncated-moment-atoms-core-variety,di18arxiv-moment-cone,dio21ma-truncated-moment-problem-caratheodory-numbers-hilbert-functions}. From an optimization viewpoint,~\cite{papp19siopt-sos-without-sdp} exploits structured matrix representations of \( \SosCone{n,2d}^* \), including Hankel-type structure in suitable bases, for direct SOS optimization, although not its facial geometry per se.

\paragraph{Atomic decomposition of representable moment matrices.}
Because every representable truncated moment sequence admits a finitely atomic representing measure~\cite{dio21ma-truncated-moment-problem-caratheodory-numbers-hilbert-functions}, one is naturally led to the problem of recovering the atoms and weights from its moment matrix. This problem appears in the truncated moment problem~\cite{dio21ma-truncated-moment-problem-caratheodory-numbers-hilbert-functions,nie14fcm-truncated-moment}, in solution extraction for moment--SOS relaxations~\cite{lasserre01siopt-global,parrilo03mp-semidefinite,henrion05ppc-detecting}, and in related topics such as cubature~\cite{fialkow05ieot-moment-multivariable-cubature} and tensor decomposition~\cite{bernardi13jsc-tensor-decomp-moment-matrices}. A standard extraction mechanism is flat extension: once a higher-degree extension has the same rank as the original truncated moment matrix, the atoms and weights can be recovered via multiplication operators or the GNS construction~\cite{henrion05ppc-detecting,klep18siopt-minimizer-extraction-robust}. This gives an a priori certificate for extraction, but with moments only up to degree \(2d\), the usual flatness test limits recovery to at most \( \binom{n+d-1}{d-1}=\Theta(n^{d-1}) \) atoms for fixed \(d\). Alternatively, one may solve higher-order SDP relaxations until flatness occurs~\cite{nie14fcm-truncated-moment}; our method instead decomposes the given degree-\(d\) moment matrix using the facial geometry of the pseudo-moment cone.

\paragraph{Efficient tensor decomposition algorithms.}
Our facial geometry results are also related to recent progress on tensor decomposition in generic regimes. Classical linear-algebraic methods include simultaneous diagonalization and its connection to tensor decomposition~\cite{de06sjmaa-multilinear-simulataneous-matrix-diagonalization,leurgans93simaa-decomposition-three-way-arrays}. Other algebraic and numerical approaches for symmetric tensors include generating-polynomial methods~\cite{nie17fcm-generatingpolynomial-symtensordecomp}. More recent works obtain generic recovery guarantees through planted rank-one matrices in suitable subspaces~\cite{johnston23focs-computing-linear-sections-varieties,dastidar25arxiv-improve-threshold-rank1-matrices}, moment-matrix extension~\cite{shi25arxiv-efficient}, and power-iteration methods~\cite{wang25arxiv-multispace-power-tensordecomp}. Several of these methods can also recover the atoms and weights of
\( X=\sum_i w_i^2 m_d(z_i)m_d(z_i)\tran \)
in generic regimes overlapping with ours, by viewing \(X\) as a structured symmetric tensor. Their mechanisms, however, are different: linear-section methods lift the problem to a polynomial space and use simultaneous diagonalization~\cite{johnston23focs-computing-linear-sections-varieties}, moment-matrix-extension methods solve for missing moments and can reduce to linear equations in favorable low-regularity regimes~\cite{shi25arxiv-efficient}, and power-method approaches recover components through singular-vector-type iterations~\cite{wang25arxiv-multispace-power-tensordecomp}. By contrast, our method uses positivity and convexity to expose extreme rays of the minimal face of \( \MomConePs{n,2d} \). This convex-geometric viewpoint yields a Carath\'eodory-type atomic decomposition and also provides a way to generate high-rank extreme rays of \( \SosCone{n,2d}^* \), which is outside the usual target of tensor-decomposition algorithms.

\section{Notation and Preliminaries}
\label{sec:pre}

In this section, we review the notation and mathematical preliminaries used throughout the paper, including background on (pseudo-)moment cones, symmetric tensors, and algebraic geometry.

We begin with some general notation. For any set $S$, let $\abs{S}$ denote its cardinality, and let $S^{\times k}$ denote the $k$-fold Cartesian product of $S$. For any positive integer $n$, let $[n] := \{1, \dots, n\}$. Throughout the paper, we endow $\Natural{n}$ with the graded lexicographic order. For a positive integer $d$, let $\MultiIndex{n}{d}$ (resp. $\MultiIndexLeq{n}{d}$) denote the set of multi-indices $\alpha = (\alpha_1, \dots, \alpha_n) \in \Natural{n}$ such that $\abs{\alpha} = d$ (resp. $\abs{\alpha} \le d$). Given $u \in \Real{n}$ and $\alpha \in \Natural{n}$, we write $u^\alpha$ for $u_1^{\alpha_1} \cdots u_n^{\alpha_n}$ and $\binom{d}{\alpha}$ for $\frac{d!}{\alpha_1! \cdots \alpha_n!}$.

Let $\calV$ be a finite-dimensional vector space. For $\calA \subseteq \calV$, let $\cone(\calA)$, $\linspan(\calA)$, $\aff(\calA)$, and $\conv(\calA)$ denote the conic hull, linear span, affine hull, and convex hull of $\calA$, respectively. For a convex set $C \subseteq \calV$, let $\ri(C)$ denote its relative interior. For any $x \in C$, let $\minface(x,C)$ denote the minimal face of $C$ containing $x$. Given $x \in \calV$, let $\ray x := \{tx \mid t \ge 0\}$ denote the ray spanned by $x$.

Given a positive integer $N$, let $\sym{N}$ denote the space of $N \times N$ real symmetric matrices, let $\psd{N}$ denote the cone of positive semidefinite (PSD) matrices in $\sym{N}$, and let $\pd{N}$ denote the cone of positive definite (PD) matrices in $\sym{N}$. Throughout the paper, we use MATLAB-style indexing for matrices: for index sets $I,J$, $A(I,J)$ denotes the submatrix of $A$ with rows indexed by $I$ and columns indexed by $J$; in particular, $A(:,J)$ and $A(I,:)$ denote the corresponding column and row sub-matrices. Given $A, B \in \sym{N}$, let $\trace{A}$ denote the trace of $A$, let $\inprod{A}{B} := \trace{AB}$ denote the trace inner product, and let $A \circ B := (A_{i,j} B_{i,j})_{i,j \in [N]}$ denote the Hadamard product of $A$ and $B$. Let $\lambda_{\max}(A)$ (resp. $\lambda_{\min}(A)$) be the maximum (resp. minimum) eigenvalue of $A$. Let $\normF{A}$ denote the Frobenius norm of $A$. For $v \in \Real{n}$, let $\normtwo{v}$ denote the Euclidean norm of $v$. For a map $f$, let $\im(f)$ denote its image. Let $I_r$ be the identity matrix of size $r$ by $r$.


\paragraph{Background on moment and pseudo-moment cones.}
Following the notation in~\cite{blekherman12jams-nonnegative-polynomial-sos}, let $\HomPoly{n,d}$ denote the space of real homogeneous polynomials of degree $d$ in $n$ variables. We write $\NonPolyCone{n,2d} := \{ p \in \HomPoly{n,2d} \mid p(x) \ge 0, \forall x \in \Real{n} \}$ and $\SosCone{n,2d} := \{ p \in \HomPoly{n,2d} \mid p = \sum_i q_i^2,\ q_i \in \HomPoly{n,d} \}$ for the cones of nonnegative forms and sums of squares (SOS), and $\NonPolyCone{n,2d}^* := \{ \ell \in \HomPoly{n,2d}^* \mid \ell(p) \ge 0, \forall p \in \NonPolyCone{n,2d} \}$ and $\SosCone{n,2d}^* := \{ \ell \in \HomPoly{n,2d}^* \mid \ell(q^2) \ge 0, \forall q \in \HomPoly{n,d} \}$ for their dual cones; we call $\SosCone{n,2d}^*$ the (homogeneous) pseudo-moment cone~\cite{henrion25arxiv-positively-not-sos}. For any $z \in \Real{n}$, let $v_d(z) := (z^\alpha)_{\alpha \in \MultiIndex{n}{d}} \in \Real{\binom{n+d-1}{d}}$ be the vector of all monomials of degree $d$, and let $m_d(z) := (z^\alpha)_{\alpha \in \MultiIndexLeq{n}{d}} \in \Real{\binom{n+d}{d}}$ be the vector of all monomials of degree at most $d$; if $z \in \Real{n}$ and $\tilde{z} := [1; z] \in \Real{n+1}$, then $m_d(z) = v_d(\tilde{z})$. For $\ell \in \HomPoly{n,2d}^*$, let $M(\ell) \in \sym{\binom{n+d-1}{d}}$ denote the matrix indexed by monomials of degree $d$, with entries $M(\ell)_{\alpha,\beta} = \ell(x^{\alpha+\beta})$ for $\alpha,\beta \in \MultiIndex{n}{d}$. Under the degree-$d$ monomial basis identification, the matrix realizations of $\NonPolyCone{n,2d}^*$ and $\SosCone{n,2d}^*$ are
\begin{subequations}\label{eq:pre:moment-cones}
\begin{align}
    \MomConeAtm{n,2d}
    &:= \{ M(\ell) \mid \ell \in \NonPolyCone{n,2d}^* \}
    = \cone\{ v_d(z) v_d(z)\tran \mid z \in \Real{n} \}, \label{eq:pre:moment-cone-atm} \\
    \MomConePs{n,2d}
    &:= \{ M(\ell) \mid \ell \in \SosCone{n,2d}^* \}
    = \Hankel{n,d} \cap \psd{\binom{n+d-1}{d}}, \label{eq:pre:moment-cone-ps}
\end{align}
\end{subequations}
where
\begin{align}
    \label{eq:pre:hankel}
    \Hankel{n,d} := \{ X \in \sym{\binom{n+d-1}{d}} \mid X_{\alpha,\beta} = X_{\alpha^\prime,\beta^\prime},\ \forall \alpha,\beta,\alpha^\prime,\beta^\prime \in \MultiIndex{n}{d} \text{ with } \alpha+\beta = \alpha^\prime+\beta^\prime \} \subseteq \sym{\binom{n+d-1}{d}}
\end{align}
is the subspace of generalized multivariate Hankel matrices. We refer to $\MomConeAtm{n,2d}$ as the cone of moment matrices with representing measures, and $\MomConePs{n,2d}$ as the cone of pseudo-moment matrices. We have $\MomConeAtm{n,2d} \subseteq \MomConePs{n,2d}$ by~\cite{blekherman12jams-nonnegative-polynomial-sos}. Finally, for any $z \in \Real{n}$, let $\ell_z \in \HomPoly{n,2d}^*$ denote the point-evaluation functional $\ell_z(p) := p(z)$. Then $M(\ell_z) = v_d(z) v_d(z)\tran$, and every rank-one extreme ray of $\SosCone{n,2d}^*$ is generated by a scaled point evaluation $c\ell_z$ with $c \ge 0$~\cite{blekherman12jams-nonnegative-polynomial-sos}.

\paragraph{Background on symmetric tensors.}
Following the notation in~\cite{johnston23focs-computing-linear-sections-varieties}, let $\calV$ be an $n$-dimensional $\mathbb{R}$-vector space with standard basis $\{e_1, \dots, e_n\}$. For $v_1, \dots, v_d \in \calV$, their ordinary tensor product is the pure tensor $v_1 \otimes \cdots \otimes v_d \in \TenProd{\calV}{d}$, where $\TenProd{\calV}{d} := \calV \otimes \cdots \otimes \calV$ denotes the $d$-fold tensor power of $\calV$. Let $\SymGroup{d}$ denote the group of permutations of $d$ elements, and let $\SymTenProd{\calV}{d} \subseteq \TenProd{\calV}{d}$ be the subspace of symmetric tensors, namely those invariant under the action of $\SymGroup{d}$ on $\TenProd{\calV}{d}$. The projection map $\SymProj{\calV}{d}: \TenProd{\calV}{d} \to \SymTenProd{\calV}{d}$ is given by
\begin{align}
    \label{eq:pre:sym-proj}
    v_1 \stimes \cdots \stimes v_d := \SymProj{\calV}{d}(v_1 \otimes \cdots \otimes v_d) = \frac{1}{d!} \sum_{\sigma \in \SymGroup{d}} v_{\sigma(1)} \otimes \cdots \otimes v_{\sigma(d)}.
\end{align}
When $v_1 = \cdots = v_d = v$, we abbreviate $v_1 \otimes \cdots \otimes v_d$ as $v^{\otimes d} = v^{\stimes d}$. The standard basis of $\calV$ induces a basis of $\SymTenProd{\calV}{d}$ given by $\{ e_1^{\stimes \alpha_1} \stimes \cdots \stimes e_n^{\stimes \alpha_n} \}_{\alpha \in \MultiIndex{n}{d}}$. Therefore, $\SymTenProd{\calV}{d} \cong \Real{\binom{n+d-1}{d}}$. In particular, $\SymTenProd{\calV}{2} \cong \Real{\binom{n+1}{2}} \cong \sym{n}$.

\paragraph{Background on algebraic geometry.}
Following~\cite{johnston23focs-computing-linear-sections-varieties}, let $\calV$ be an $n$-dimensional $\mathbb{R}$-vector space. A subset $\calX \subseteq \calV$ is a variety if there exist polynomials $f_1, \dots, f_p \in \mathbb{R}[x_1, \dots, x_n]$ such that $\calX = \{v \in \calV \mid f_1(v) = \cdots = f_p(v) = 0\}$.
The Zariski topology on $\calV$ is the topology whose closed sets are precisely these varieties. If $\calX \subseteq \calV$ is a variety, then $\calX$ carries the induced subspace topology: $\calU \subseteq \calX$ is open iff $\calU = \calX \cap \calO$ for some Zariski-open subset $\calO \subseteq \calV$~\cite[Chapter 1]{hartshorne13springer-algebraic-geometry}. For any $\calA \subseteq \calV$, let $\clE(\calA)$ and $\clZ(\calA)$ denote its Euclidean and Zariski closures, respectively. If $\calA \subseteq \calX$ for a variety $\calX$, let $\clZ[\calX](\calA)$ denote the closure of $\calA$ in the induced Zariski topology on $\calX$; since $\calX$ is Zariski closed, $\clZ(\calA) = \clZ[\calX](\calA)$. A variety $\calX$ is conic if it is cut out by homogeneous polynomials. For an irreducible conic variety $\calX \subseteq \calV$, non-degeneracy of order $d$ means that no nonzero homogeneous polynomial of degree $d$ vanishes on $\calX$; in particular, non-degeneracy of order $1$ is equivalent to $\linspan(\calX) = \calV$. For example,
\begin{align}
    \label{eq:pre:sym-prod-tensor}
    \SymRankOne := \{c v^{\otimes d} \mid c \in \mathbb{R}, v \in \calV \} \subseteq \SymTenProd{\calV}{d}
\end{align}
is the set of symmetric product tensors in $\SymTenProd{\calV}{d}$, and it is non-degenerate of order $1$~\cite[Section 2.3]{johnston23focs-computing-linear-sections-varieties}. Finally, a property $P : \calX \to \{0,1\}$ holds \emph{generically} on a variety $\calX$ if there exists a Zariski-open dense subset $\calA \subseteq \calX$ in the induced subspace topology such that $P(x) = 1$ for all $x \in \calA$. In particular, the set on which $P$ holds is Euclidean dense in $\calX$~\cite[Section 2.2]{johnston23focs-computing-linear-sections-varieties}. This notion extends naturally to Cartesian products of varieties~\cite[Section 2.2]{johnston23focs-computing-linear-sections-varieties}: if $\calX_1, \dots, \calX_r \subseteq \calV$ are varieties, then $\calX_1 \times \cdots \times \calX_r \subseteq \calV^{\times r}$ is again a variety. We say that a property $P$ holds \emph{generically} on $\calX_1 \times \cdots \times \calX_r$ if there exists a Zariski-open dense subset $\calB \subseteq \calX_1 \times \cdots \times \calX_r$ such that $P(v_1, \dots, v_r)$ holds for all $(v_1, \dots, v_r) \in \calB$, where $v_i \in \calX_i$ for each $i \in [r]$.


\section{Simplicial Regularizability of the Pseudo-Moment Cone}
\label{sec:sim}

In this section, we aim to establish the following theorem.
\begin{theorem}[Simplicial regularizability of the pseudo-moment cone]
    \label{thm:sim:minface-condition}
    Given positive integers $n, d$, define $\MomConePs{n+1,2d}$ as in~\eqref{eq:pre:moment-cone-ps}. Suppose $s$ is a positive integer such that
    \begin{align}
        \label{eq:sim:minface-s}
        s \le \frac{1}{2} \left( \binom{n+d}{d} + 1 \right) - \binom{n+d}{d}^{-1} \binom{n+2d}{2d}.
    \end{align}
    For fixed $d \ge 2$, the upper bound in~\eqref{eq:sim:minface-s} grows as $\Theta(n^d)$. 
    Then, the following geometric condition
    \begin{align}
        \label{eq:sim:minface-condition}
        \minface \left( \sum_{i=1}^s w_i^2 m_d(z_i) m_d(z_i)\tran, \MomConePs{n+1,2d} \right) = \cone\left( \left\{m_d(z_i) m_d(z_i)\tran \right\}_{i=1}^s \right)
    \end{align}
    holds for generically chosen $\{(w_i, z_i)\}_{i=1}^s \in (\reals \times \Real{n})^{\times s}$, with $\{m_d(z_i) m_d(z_i)\tran \}_{i=1}^s$ linearly independent.
\end{theorem}

Our proof relies on the following key result from~\cite{johnston23focs-computing-linear-sections-varieties}.
\begin{theorem}[\titlemath{Generic mixed-term avoidance~\cite[Theorem 15]{johnston23focs-computing-linear-sections-varieties}}]
    \label{thm:sim:calK-dim}
    Let $\calV$ be an $N$-dimensional $\reals$-vector space, $\kappa, r$ be positive integers, $s \in [r]$, and $\calK \subseteq \SymTenProd{\calV}{\kappa}$ a linear subspace. Let $\calX_1, \dots, \calX_r \subseteq \calV$ be conic varieties for which $\calX_1, \dots, \calX_s$ are non-degenerate of order $\kappa - 1$ and $\calX_{s+1}, \dots, \calX_r$ are non-degenerate of order $\kappa$. If 
    \begin{align}
        \label{eq:sim:calK-dim}
        \codim(\calK) \ge r(\kappa - 1)! \binom{
            N+\kappa-2
        }{\kappa-1},
    \end{align}
    then for generically chosen $v_1 \in \calX_1, \dots, v_r \in \calX_r$, it holds that 
    \begin{align}
        \label{eq:sim:linspace-variety-intersection}
        \linspan\left\{ 
            v_1^{\stimes \alpha_1} \stimes \cdots \stimes v_r^{\stimes \alpha_r} \mymid \alpha \in \MultiIndex{r}{\kappa} \backslash \Delta_s
         \right\} \cap \calK = \{0\},
    \end{align}
    where $\Delta_s := \{ \kappa e_i \mid i \in [s] \} \subset \Natural{r}$, and $e_i$ denotes the $i$-th standard basis vector in $\Real{r}$.
\end{theorem}

Before stepping into the proof of Theorem~\ref{thm:sim:minface-condition}, we provide several elementary lemmas.
We start by two lemmas on generic properties.

\begin{lemma}[Generic property under polynomial mapping]
    \label{lem:sim:generic-polymap}
    Let $\calV$ and $\calW$ be two finite-dimensional $\reals$-vector spaces, $\calX \subseteq \calV$ a variety, and $P$ a property that holds generically on $\calX$. Suppose $f : \calW \to \calV$ is a polynomial mapping such that $f(\calW) \subseteq \calX$ and $f(\calW)$ is Zariski dense in $\calX$. Then $P \circ f$ holds generically on $\calW$.
\end{lemma}
\begin{proof}
    Let $\calA \subseteq \calX$ be a Zariski-open dense subset on which $P$ holds. Since $f(\calW) \subseteq \calX$, the preimage $f^{-1}(\calA)$ is well-defined. Moreover, because $f$ is a polynomial mapping, it is continuous in the Zariski topology, so $f^{-1}(\calA)$ is Zariski open in $\calW$. Since $f(\calW)$ is Zariski dense in $\calX$ and $\calA$ is Zariski open dense in $\calX$, we have $f(\calW) \cap \calA \neq \emptyset$, and hence $f^{-1}(\calA) \neq \emptyset$. Finally, because $\calW$ is irreducible in the Zariski topology, every nonempty Zariski-open subset of $\calW$ is dense. Therefore, $f^{-1}(\calA)$ is Zariski open dense in $\calW$, which shows that $P \circ f$ holds generically on $\calW$.
\end{proof}

\begin{lemma}[Generic property under finite conjunction]
    \label{lem:sim:generic-conjuction}
    Let $\calV$ be a finite-dimensional $\reals$-vector space and $\calX \subseteq \calV$ an irreducible variety. If properties $P_1$ and $P_2$ both hold generically on $\calX$, then $P_1 \wedge P_2$ also holds generically on $\calX$.
\end{lemma}
\begin{proof}
    By definition, there exist Zariski-open dense subsets $\calA_1, \calA_2 \subseteq \calX$ such that $P_1$ holds on $\calA_1$ and $P_2$ holds on $\calA_2$. Hence, $\calA_1 \cap \calA_2$ is Zariski open in $\calX$. Since $\calX$ is irreducible, every nonempty Zariski-open subset of $\calX$ is dense, so $\calA_1 \cap \calA_2$ is Zariski open dense in $\calX$. On this set, both $P_1$ and $P_2$ hold, and therefore $P_1 \wedge P_2$ holds generically on $\calX$.
\end{proof}

We then provide two auxiliary lemmas related to symmetric tensors.
\begin{lemma}[Isomorphism of symmetric square]
    \label{lem:sim:iso-sym-square}
    Let $\calV$ be a finite-dimensional $\reals$-vector space, and let $\phi : \calV \to \Real{N}$ be an isomorphism. Then there exists a unique linear isomorphism $\psi : \SymTenProd{\calV}{2} \to \sym{N}$ such that
    \begin{align}
        \label{eq:sim:iso-sym-square}
        \psi(x \stimes y) = \frac{1}{2} \phi(x) \phi(y)\tran + \frac{1}{2} \phi(y) \phi(x)\tran, \quad \forall x, y \in \calV.
    \end{align}
\end{lemma}
\begin{proof}
    Define $f(x, y) := \frac{1}{2} \phi(x) \phi(y)\tran + \frac{1}{2} \phi(y) \phi(x)\tran,  \forall x, y \in \calV.$
    Then $f$ is bilinear and symmetric. Hence, by the universal property of the symmetric square~\cite[Theorem 2.24]{gallier20springer-differential-geometry-lie-group}, $f$ induces a unique linear map $\psi : \SymTenProd{\calV}{2} \to \sym{N}$ satisfying~\eqref{eq:sim:iso-sym-square}.
    Since $\phi$ is an isomorphism from $\calV$ to $\Real{N}$, we may choose a basis $\{\varepsilon_i\}_{i=1}^{N}$ of $\calV$ and let $\{e_i\}_{i=1}^{N}$ be the standard basis of $\Real{N}$ such that $\phi(\varepsilon_i) = e_i$ for all $i \in [N]$. Then
    \begin{align*}
        \psi(\varepsilon_i \stimes \varepsilon_i) = e_i e_i\tran, \quad
        \psi(\varepsilon_i \stimes \varepsilon_j) = \frac{1}{2} (e_i e_j\tran + e_j e_i\tran), \quad i < j.
    \end{align*}
    Therefore, $\psi$ sends a basis of $\SymTenProd{\calV}{2}$ to a basis of $\sym{N}$, and hence $\psi$ is an isomorphism.
\end{proof}

\begin{lemma}[Expansion of rank-one symmetric tensor]
    \label{lem:sim:expansion-rank1-symtensor}
    Let $u \in \Real{n}$, and $\{e_i\}_{i=1}^n$ be the standard basis of $\Real{n}$. Then, for a positive integer $d$,
    \begin{align}
        \label{eq:sim:expansion-rank1-symtensor}
        u^{\otimes d} = \sum_{\alpha \in \MultiIndex{n}{d}} u^\alpha \binom{d}{\alpha} (e_1^{\stimes \alpha_1} \stimes \cdots \stimes e_n^{\stimes \alpha_n}).
    \end{align}
\end{lemma}
\begin{proof}
    For any $\alpha \in \MultiIndex{n}{d}$, define 
    $
        I(\alpha) := \left\{  (i_1, \dots, i_d) \in [n]^{\times d} : \abs{\{r \mid i_r = j\}} = \alpha_j, \forall j \in [n] \right\}
    $, \ie the set of all $d$-tuples of indices that contain exactly $\alpha_j$ occurrences of $j$ for each $j \in [n]$.
    Then from~\eqref{eq:pre:sym-proj}, 
    \begin{align*}
        e_1^{\stimes \alpha_1} \stimes \cdots \stimes e_n^{\stimes \alpha_n} = \frac{\alpha_1! \cdots \alpha_n!}{d!} \sum_{(i_1, \dots, i_d) \in I(\alpha)} e_{i_1} \otimes \cdots \otimes e_{i_d} = \binom{d}{\alpha}^{-1} \sum_{(i_1, \dots, i_d) \in I(\alpha)} e_{i_1} \otimes \cdots \otimes e_{i_d}, \ \forall \alpha \in \MultiIndex{n}{d}.
    \end{align*}
    On the other hand, since $u = \sum_{i=1}^n u_i e_i$, we have 
    \begin{align*}
        u^{\otimes d} = \sum_{(i_1,\dots,i_d) \in [n]^{\times d}} u_{i_1} \cdots u_{i_d} e_{i_1} \otimes \cdots \otimes e_{i_d} = \sum_{\alpha \in \MultiIndex{n}{d}} u^\alpha \sum_{(i_1, \dots, i_d) \in I(\alpha)} e_{i_1} \otimes \cdots \otimes e_{i_d}. 
    \end{align*}
    Combining the above two equations, we have established~\eqref{eq:sim:expansion-rank1-symtensor}. 
\end{proof}

Finally, we provide three lemmas on linear algebra and convex geometry. 
\begin{lemma}[Intersection between subspaces]
    \label{lem:sim:intersection-subspaces}
    Let $\calV$ be a finite-dimensional $\reals$-vector space, and $\calL_1, \calL_2, \calK \subseteq \calV$ be three subspaces. If $\calL_1 \cap \calK = \{0\}$ and $\calL_2 \subseteq \calK$, then $(\calL_1 + \calL_2) \cap \calK = \calL_2$.   
\end{lemma}
\begin{proof}
    If $x \in (\calL_1 + \calL_2) \cap \calK$, then $x = l_1 + l_2$ for some $l_1 \in \calL_1$ and $l_2 \in \calL_2$. Since $x \in \calK$ and $l_2 \in \calK$, $l_1 = x - l_2 \in \calK$. But $l_1 \in \calL_1$, so $l_1 \in \calL_1 \cap \calK = 0$. Hence, $x = l_2 \in \calL_2$ and $(\calL_1 + \calL_2) \cap \calK \subseteq \calL_2$. The reverse inclusion is immediate because $\calL_2 \subseteq \calL_1 + \calL_2$ and $\calL_2 \subseteq \calK$. 
\end{proof}

\begin{lemma}[Facial structure of the PSD cone]
    \label{lem:sim:V-sym-Vt-psd}
    Let $V \in \Real{N \times r}$ with $r \le N$ be of full column rank. Then
    $V \sym{r} V\tran \cap \psd{N} = V \psd{r} V\tran$, where $V \sym{r} V\tran := \{V C V\tran \mid C \in \sym{r}\}$ and $V \psd{r} V\tran := \{V C V\tran \mid C \in \psd{r}\}$.
\end{lemma}
\begin{proof}
    The inclusion $V \psd{r} V\tran \subseteq V \sym{r} V\tran \cap \psd{N}$ is immediate. Conversely, let $X \in V \sym{r} V\tran \cap \psd{N}$. Then $X = V C V\tran$ for some $C \in \sym{r}$. Since $V$ has full column rank, $V\tran : \Real{N} \to \Real{r}$ is surjective. Hence, for any $y \in \Real{r}$, there exists $x \in \Real{N}$ such that $V\tran x = y$. Therefore,
    $y\tran C y = x\tran V C V\tran x = x\tran X x \ge 0$,
    because $X \in \psd{N}$. Thus $C \in \psd{r}$, so $X \in V \psd{r} V\tran$. This proves the result.
\end{proof}

\begin{lemma}[Minimal faces in the PSD cone]
    \label{lem:sim:minimal-faces-psd-cone}
    If $X = V \Lambda V\tran$ with $V \in \Real{N \times r}$ ($r \le N$) of full column rank and $\Lambda \in \pd{r}$, then $\minface(X, \psd{N}) = V \psd{r} V\tran$.
\end{lemma}
\begin{proof}
    Let $V = Q R$ be the thin QR factorization of $V$, where $Q \in \Real{N \times r}$ has orthonormal columns and $R \in \Real{r \times r}$ is invertible. Then $X = V \Lambda V\tran = Q (R \Lambda R\tran) Q\tran$. Since $\Lambda \in \pd{r}$ and $R$ is invertible, we have $R \Lambda R\tran \in \pd{r}$, and hence $X \in \ri(Q \psd{r} Q\tran)$. By~\cite[Corollary 6.1]{lewis98laa-eigenvalue-constrained-faces}, it follows that $\minface(X, \psd{N}) = Q \psd{r} Q\tran$. It remains to show that $Q \psd{r} Q\tran = V \psd{r} V\tran$. For any $W \in \psd{r}$, $Q W Q\tran = V \bigl(R^{-1} W (R\tran)^{-1}\bigr) V\tran$, and $R^{-1} W (R\tran)^{-1} \in \psd{r}$. Hence $Q \psd{r} Q\tran \subseteq V \psd{r} V\tran$. Conversely, for any $Y \in \psd{r}$, $V Y V\tran = Q (R Y R\tran) Q\tran$, and $R Y R\tran \in \psd{r}$. Therefore $V \psd{r} V\tran \subseteq Q \psd{r} Q\tran$. Thus $Q \psd{r} Q\tran = V \psd{r} V\tran$, and the result follows.
\end{proof}

Now we are ready to prove Theorem~\ref{thm:sim:minface-condition}. 
\begin{proof}[Proof of Theorem~\ref{thm:sim:minface-condition}]
    
    We prove the theorem in five steps. First, we apply Theorem~\ref{thm:sim:calK-dim} with \(\calV=\SymTenProd{\Real{n+1}}{d}\) and \(\kappa=2\), obtaining a generic avoidance statement for the mixed tensors \(v_i\vee v_j\), \(i<j\). Second, we pull this statement back to the affine parameters \((w_i,z_i)\) via the polynomial map \((w_i,z_i)\mapsto w_i\tilde z_i^{\otimes d}\), where \(\tilde z_i=[1;z_i]\). Third, we define \(\iota:\SymTenProd{\Real{n+1}}{2d}\to\SymTenProd{\SymTenProd{\Real{n+1}}{d}}{2}\) and set \(\calK=\im(\iota)\). Its image contains the diagonal terms, while the mixed terms generically avoid it; hence the intersection of the full pairwise span with \(\calK\) is exactly the diagonal span. The injectivity of \(\iota\) gives the stated bound on \(s\). Fourth, we identify \(\SymTenProd{\SymTenProd{\Real{n+1}}{d}}{2}\) with a symmetric matrix space, under which \(\calK\) becomes the generalized multivariate Hankel subspace and the diagonal terms become \(m_d(z_i)m_d(z_i)\tran\). Finally, after proving generic linear independence of the moment vectors \(m_d(z_i)\), we intersect with the PSD cone to obtain the claimed simplicial minimal face.

    \textbf{Step 1.} Fix two positive integers $n$ and $d$. 
    From Theorem~\ref{thm:sim:calK-dim}, take $\calV = \SymTenProd{\Real{n+1}}{d}$, $\kappa = 2$, and $s = r$. Take $\calX_1 = \cdots = \calX_s = \SymRankOne$ as defined in~\eqref{eq:pre:sym-prod-tensor}. The assumptions of Theorem~\ref{thm:sim:calK-dim} are all satisfied, since $\SymRankOne$ is non-degenerate of order $1 = \kappa - 1$. Therefore,~\eqref{eq:sim:calK-dim} can be simplified as 
    \begin{align}
        \label{eq:sim:minface-proof-dimK}
        \codim(\calK) \ge s \binom{n+d}{d},
    \end{align}
    with $\calK \subseteq \SymTenProd{\SymTenProd{\Real{n+1}}{d}}{2}$,
    and~\eqref{eq:sim:linspace-variety-intersection} can be simplified as
    \begin{align}
        \label{eq:sim:minface-proof-intersection} 
        \linspan \left\{ 
            v_i \stimes v_j \mymid i < j, i, j \in [s]
         \right\} \cap \calK = \{0\}
    \end{align} 
    with $\{v_i\}_{i=1}^s$ generically chosen from $(\SymRankOne)^{\times s} \subseteq (\SymTenProd{\Real{n+1}}{d})^{\times s}$.

    \textbf{Step 2.}
    Let $\{e_i\}_{i=1}^{n+1}$ be the standard basis of $\Real{n+1}$, $\{e_\alpha\}_{\alpha \in \MultiIndex{n+1}{d}}$ be the standard basis of $\Real{\binom{n+d}{d}}$, and $\{e_\gamma\}_{\gamma \in \MultiIndex{n+1}{2d}}$ be the standard basis of $\Real{\binom{n+2d}{2d}}$. Define
    $\{b_{\alpha} := \binom{d}{\alpha} e_1^{\stimes \alpha_1} \stimes \cdots \stimes e_{n+1}^{\stimes \alpha_{n+1}} \mid \alpha \in \MultiIndex{n+1}{d}\}$
    as a basis of $\SymTenProd{\Real{n+1}}{d}$. By Lemma~\ref{lem:sim:expansion-rank1-symtensor}, for any $u \in \Real{n+1}$,
    $u^{\otimes d} = \sum_{\alpha \in \MultiIndex{n+1}{d}} u^{\alpha} b_{\alpha}$.
    Similarly, define
    $\{c_{\gamma} := \binom{2d}{\gamma} e_1^{\stimes \gamma_1} \stimes \cdots \stimes e_{n+1}^{\stimes \gamma_{n+1}} \mid \gamma \in \MultiIndex{n+1}{2d}\}$
    as a basis of $\SymTenProd{\Real{n+1}}{2d}$, so that
    $u^{\otimes 2d} = \sum_{\gamma \in \MultiIndex{n+1}{2d}} u^\gamma c_\gamma$.

    Construct $f: (\reals \times \Real{n})^{\times s} \to (\SymTenProd{\Real{n+1}}{d})^{\times s}$ as the map defined by
    $f(\{w_i, z_i\}_{i=1}^{s}) = \{w_i \tz_i^{\otimes d}\}_{i=1}^{s}$ for all $\{w_i, z_i\}_{i=1}^{s} \in (\reals \times \Real{n})^{\times s}$, where $\tz := [1; z] \in \Real{n+1}$. Under the basis $\{b_\alpha\}_{\alpha \in \MultiIndex{n+1}{d}}$, $w \tz^{\otimes d} = w \sum_{\alpha \in \MultiIndex{n+1}{d}} \tz^\alpha b_\alpha$. By the fact that the class of polynomial maps between finite-dimensional vector spaces is closed under taking products, $f$ is a polynomial map. Since $w \tz^{\otimes d} \in \SymRankOne$ for every $w \in \reals$ and $z \in \Real{n}$, we have $f((\reals \times \Real{n})^{\times s}) \subseteq (\SymRankOne)^{\times s}$.
    We next show that $f((\reals \times \Real{n})^{\times s})$ is Zariski dense in $(\SymRankOne)^{\times s}$.

    Take any $v = \{w_i u_i^{\otimes d}\}_{i=1}^{s} \in (\SymRankOne)^{\times s}$, where $w_i \in \reals$ and $u_i \in \Real{n+1}$. After reordering the indices if necessary, we may assume that there exists $k \in \{0,1,\dots,s\}$ such that $u_{i,1} = 0$ for $i = 1,\dots,k$, while $u_{i,1} \ne 0$ for $i = k+1,\dots,s$. For each $t > 0$, define $v(t) := \{g_i(t)\}_{i=1}^{s}$ by
    \begin{equation}
        g_i(t) := \begin{cases}
            w_i [t; u_{i,2}; \dots; u_{i,n+1}]^{\otimes d} = \displaystyle w_i t^d \left[1; \frac{u_{i,2}}{t}; \dots; \frac{u_{i,n+1}}{t}\right]^{\otimes d} &  \text{if } i = 1,\dots,k, \\
            w_i [u_{i,1}; u_{i,2}; \dots; u_{i,n+1}]^{\otimes d} = \displaystyle w_i u_{i,1}^d \left[1; \frac{u_{i,2}}{u_{i,1}}; \dots; \frac{u_{i,n+1}}{u_{i,1}}\right]^{\otimes d} &  \text{if } i = k+1,\dots,s.
        \end{cases}
    \end{equation}
    For every $t > 0$, we have $v(t) \in f((\reals \times \Real{n})^{\times s})$. Moreover, $g_i(t) \to w_i u_i^{\otimes d}$ as $t \downarrow 0$ for each $i \in [s]$, and hence $v(t) \to v$. Therefore,
    $\clE(f((\reals \times \Real{n})^{\times s})) = (\SymRankOne)^{\times s}$.
    Since $\clE(f((\reals \times \Real{n})^{\times s})) \subseteq \clZ(f((\reals \times \Real{n})^{\times s}))$ and $(\SymRankOne)^{\times s}$ is Zariski closed, it follows that
    $\clZ(f((\reals \times \Real{n})^{\times s})) = (\SymRankOne)^{\times s}$.
    Hence by Lemma~\ref{lem:sim:generic-polymap} and~\eqref{eq:sim:minface-proof-intersection}, the following condition holds for generically chosen $\{w_i, z_i\}_{i=1}^{s} \in (\reals \times \Real{n})^{\times s}$ for $s$ satisfying~\eqref{eq:sim:minface-proof-dimK}:
    \begin{align}
        \label{eq:sim:minface-proof-ploymap}
        \linspan \left\{ 
            (w_i \tz_i^{\otimes d}) \stimes (w_j \tz_j^{\otimes d}) \mymid i < j, i, j \in [s]
         \right\} \cap \calK = \{0\}.
    \end{align} 
    
    \textbf{Step 3.}
    Let $\iota: \SymTenProd{\Real{n+1}}{2d} \to \SymTenProd{\SymTenProd{\Real{n+1}}{d}}{2}$ be the linear map defined by
    \begin{align}
        \label{eq:sim:minface-proof-iota}
        \iota(c_\gamma) = \sum_{\substack{
            \alpha + \beta = \gamma, \alpha \le \beta \\
            \alpha, \beta \in \MultiIndex{n+1}{d}
        }} (2 - \delta_{\alpha, \beta}) b_{\alpha} \stimes b_{\beta}, \quad \forall \gamma \in \MultiIndex{n+1}{2d},
    \end{align}
    where $\delta_{\alpha, \beta} = 1$ if $\alpha = \beta$ and $0$ otherwise. We choose $\calK$ in~\eqref{eq:sim:minface-proof-ploymap} to be $\im(\iota)$.

    We first show that $\iota$ is injective. Suppose $p, q \in \SymTenProd{\Real{n+1}}{2d}$ satisfy $\iota(p) = \iota(q)$. Writing $p = \sum_{\gamma} p_\gamma c_\gamma$ and $q = \sum_{\gamma} q_\gamma c_\gamma$, we have
    \begin{align*}
        \sum_{\gamma \in \MultiIndex{n+1}{2d}} \sum_{\substack{
            \alpha + \beta = \gamma, \alpha \le \beta \\
            \alpha, \beta \in \MultiIndex{n+1}{d}
        }} (2 - \delta_{\alpha, \beta}) (p_\gamma - q_\gamma) b_{\alpha} \stimes b_{\beta}
        = \sum_{\substack{
            \alpha \le \beta \\
            \alpha, \beta \in \MultiIndex{n+1}{d}
        }} (2 - \delta_{\alpha, \beta}) (p_{\alpha+\beta} - q_{\alpha+\beta}) b_{\alpha} \stimes b_{\beta} = 0.
    \end{align*}
    Since $\{b_\alpha \stimes b_\beta \mid \alpha \le \beta\}$ is a basis of $\SymTenProd{\SymTenProd{\Real{n+1}}{d}}{2}$, it follows that $p_{\alpha+\beta} = q_{\alpha+\beta}$ for all $\alpha, \beta$, and hence $p = q$. Therefore, $\iota$ is injective, so
    $\dim(\im(\iota)) = \dim(\SymTenProd{\Real{n+1}}{2d}) = \binom{n+2d}{2d}$.
    Consequently,~\eqref{eq:sim:minface-proof-dimK} is satisfied whenever
    \begin{align}
        \label{eq:sim:minface-proof-s}
        s \le \binom{n+d}{d}^{-1} \left\{ \dim(\SymTenProd{\SymTenProd{\Real{n+1}}{d}}{2}) - \dim(\im(\iota)) \right\}
        = \frac{1}{2} \left( \binom{n+d}{d} + 1 \right) - \binom{n+d}{d}^{-1} \binom{n+2d}{2d}.
    \end{align}

    Next, for any $w \in \reals$ and $u \in \Real{n+1}$, we have
    \begin{align*} 
        & \iota( w^2 u^{\otimes 2d} ) = w^2 \sum_{\gamma \in \MultiIndex{n+1}{2d}} u^\gamma \iota(c_\gamma) 
        = w^2 \sum_{\gamma \in \MultiIndex{n+1}{2d}} \sum_{\substack{ \alpha + \beta = \gamma, \alpha \le \beta \\ \alpha, \beta \in \MultiIndex{n+1}{d} }} u^\gamma (2 - \delta_{\alpha, \beta}) b_{\alpha} \stimes b_{\beta} \\ 
        = & w^2 \sum_{\alpha \in \MultiIndex{n+1}{d}} u^{2\alpha} b_\alpha \stimes b_\alpha + 2 w^2 \sum_{\substack{\alpha < \beta \\ \alpha, \beta \in \MultiIndex{n+1}{d}}} u^{\alpha+\beta} b_\alpha \stimes b_\beta 
        = w^2 \left(\sum_{\alpha \in \MultiIndex{n+1}{d}} u^\alpha b_\alpha \right) \stimes \left(\sum_{\beta \in \MultiIndex{n+1}{d}} u^\beta b_\beta \right) 
        = (w u^{\otimes d}) \stimes (w u^{\otimes d}). 
    \end{align*}
    Hence
    $\linspan(\{ (w_i \tz_i^{\otimes d}) \stimes (w_i \tz_i^{\otimes d}) \}_{i=1}^s) \subseteq \im(\iota)$.
    Combining this with Lemma~\ref{lem:sim:intersection-subspaces} and~\eqref{eq:sim:minface-proof-intersection}, we conclude that
    \begin{align}
        \label{eq:sim:minface-proof-leq}
        \linspan \left( \left\{
            (w_i \tz_i^{\otimes d}) \stimes (w_j \tz_j^{\otimes d}) \mymid i \le j, i, j \in [s]
        \right\} \right) \cap \im(\iota) 
        = \linspan \left( \left\{ (w_i \tz_i^{\otimes d}) \stimes (w_i \tz_i^{\otimes d}) \mymid i \in [s] \right\} \right)
    \end{align}
    for generically chosen $\{w_i, z_i\}_{i=1}^{s} \in (\reals \times \Real{n})^{\times s}$ with $s$ satisfying~\eqref{eq:sim:minface-proof-s}.

    \textbf{Step 4.} Since $\SymTenProd{\SymTenProd{\Real{n+1}}{d}}{2} \cong \sym{\binom{n+d}{d}}$, we now map every object in~\eqref{eq:sim:minface-proof-leq} back to $\sym{\binom{n+d}{d}}$. Define an isomorphism $\phi$ between $\SymTenProd{\Real{n+1}}{d}$ and $\Real{\binom{n+d}{d}}$ by $\phi(b_\alpha) = e_\alpha$ for all $\alpha \in \MultiIndex{n+1}{d}$. Hence $\phi(u^{\otimes d}) = \sum_{\alpha \in \MultiIndex{n+1}{d}} u^\alpha e_\alpha = v_d(u)$ for any $u \in \Real{n+1}$.
    By Lemma~\ref{lem:sim:iso-sym-square}, we can construct an isomorphism $\psi$ from $\SymTenProd{\SymTenProd{\Real{n+1}}{d}}{2}$ to $\sym{\binom{n+d}{d}}$ by
    \begin{align}
        \label{eq:sim:minface-proof-psi}
        \psi(b_\alpha \stimes b_\beta) = \frac{1}{2} \phi(b_\alpha) \phi(b_\beta)\tran + \frac{1}{2} \phi(b_\beta) \phi(b_\alpha)\tran = \frac{1}{2}(e_\alpha e_\beta\tran + e_\beta e_\alpha\tran ) = \frac{1}{2} (E_{\alpha, \beta} + E_{\beta, \alpha})
    \end{align}
    for all $\alpha \le \beta, \alpha, \beta \in \MultiIndex{n+1}{d}$, where $E_{\alpha, \beta} \in \sym{\binom{n+d}{d}}$ denotes the matrix with a ``$1$'' in the $(\alpha, \beta)$-th entry and ``$0$'' in all other entries.

    It follows from~\eqref{eq:sim:minface-proof-iota} and~\eqref{eq:sim:minface-proof-psi} that, for any $p = \sum_{\gamma \in \MultiIndex{n+1}{2d}} p_{\gamma} c_\gamma \in \SymTenProd{\Real{n+1}}{2d}$, $\psi(\iota(p))$ has the following form:
    \begin{align*}
        \sum_{\gamma \in \MultiIndex{n+1}{2d}} \sum_{\substack{
            \alpha + \beta = \gamma, \alpha \le \beta \\
            \alpha, \beta \in \MultiIndex{n+1}{d}
        }} (2 - \delta_{\alpha, \beta}) p_\gamma \psi(b_{\alpha} \stimes b_{\beta})
        = \sum_{\substack{
            \alpha \le \beta \\
            \alpha, \beta \in \MultiIndex{n+1}{d}
        }} \frac{2 - \delta_{\alpha, \beta}}{2} p_{\alpha + \beta} (E_{\alpha, \beta} + E_{\beta, \alpha})
        = \sum_{\alpha, \beta \in \MultiIndex{n+1}{d}} p_{\alpha + \beta} E_{\alpha, \beta}.
    \end{align*}
    Therefore, $\psi(\im(\iota)) = \left\{\sum_{\alpha, \beta \in \MultiIndex{n+1}{d}} y_{\alpha + \beta} E_{\alpha, \beta} \mid y \in \Real{\binom{n+2d}{2d}} \right\}$. Recall the definition of the set of generalized multivariate Hankel matrices $\Hankel{n+1,d} \subseteq \sym{\binom{n+d}{d}}$ in~\eqref{eq:pre:hankel}. For any $X \in \psi(\im(\iota))$, we have $X_{\alpha, \beta} = X_{\alpha^\prime, \beta^\prime} = y_{\alpha+\beta}$ whenever $\alpha + \beta = \alpha^\prime + \beta^\prime$. Hence $\psi(\im(\iota)) \subseteq \Hankel{n+1,d}$. Conversely, suppose that $X \in \Hankel{n+1,d}$. By setting $y_{\alpha+\beta} = X_{\alpha, \beta}$, we obtain $X = \sum_{\alpha,\beta} y_{\alpha+\beta} E_{\alpha, \beta} \in \psi(\im(\iota))$. Hence $\psi(\im(\iota)) = \Hankel{n+1,d}$.
    
    For $(w_i \tz_i^{\otimes d}) \stimes (w_j \tz_j^{\otimes d})$, $\forall i \le j, i, j \in [s]$, 
    \begin{align*}
        & \psi((w_i \tz_i^{\otimes d}) \stimes (w_j \tz_j^{\otimes d})) 
        = \frac{1}{2} w_i w_j \left\{ 
            \phi(\tz_i^{\otimes d}) \phi(\tz_j^{\otimes d})\tran + \phi(\tz_j^{\otimes d}) \phi(\tz_i^{\otimes d})\tran
         \right\} \\
        = & \frac{1}{2} w_i w_j \left\{ v_d(\tz_i) v_d(\tz_j)\tran + v_d(\tz_j) v_d(\tz_i)\tran \right\}
        = \frac{1}{2} w_i w_j \left\{ m_d(z_i) m_d(z_j)\tran + m_d(z_j) m_d(z_i)\tran \right\}  .
    \end{align*}
    Therefore, $\psi(\linspan ( \{ (w_i \tz_i^{\otimes d}) \stimes (w_i \tz_i^{\otimes d}) \mid i \in [s] \} )) = \left\{\sum_{i = 1}^s c_i w_i^2 m_d(z_i) m_d(z_i)\tran \mid c_i \in \reals, \forall i \in [s] \right\}$, and 
    \begin{align}
        \label{eq:sim:minface-proof-VdWC}
        & \psi\left(\linspan \left( \left\{
            (w_i \tz_i^{\otimes d}) \stimes (w_j \tz_j^{\otimes d}) \mymid i \le j, i, j \in [s]
        \right\} \right)\right) \nonumber \\
        = & \left\{ 
            \sum_{1 \le i \le j \le s} \frac{1}{2} w_i w_j c_{i,j} \left( 
                m_d(z_i) m_d(z_j)\tran + m_d(z_j) m_d(z_i)\tran 
             \right) \mymid c_{i,j} \in \reals, \forall 1 \le i \le j \le s
         \right\} \nonumber \\
        = & \left\{ 
            V_d(Z) \left(\frac{1}{2}W \circ C \right) V_d(Z)\tran \mymid C \in \sym{s}
         \right\},
    \end{align}
    where, in the last equality, $Z := [z_1, \dots, z_s] \in \Real{n \times s}$, $V_d(Z) := [m_d(z_1), \dots, m_d(z_s)] \in \Real{\binom{n+d}{d} \times s}$, and $W \in \sym{s}$ satisfies $W_{i,j} = w_i w_j$ for all $i, j \in [s]$. Additionally, we note that for generically chosen $\{(w_i, z_i)\}_{i=1}^s \in (\reals \times \Real{n})^{\times s}$, the following condition holds:
    \begin{align}
        \label{eq:sim:minface-proof-nonzero}
        w_i \ne 0, \quad \forall i \in [s].
    \end{align}
    Indeed, the failure set is $\{\{(w_i, z_i)\}_{i=1}^s \mid \Pi_{i=1}^s w_i = 0\}$, which is a proper sub-variety of $(\reals \times \Real{n})^{\times s}$. Under~\eqref{eq:sim:minface-proof-nonzero}, $\frac{1}{2}W$ in~\eqref{eq:sim:minface-proof-VdWC} is nonzero entry-wise and can therefore be absorbed into $C$:
    $
        \psi(\linspan ( \{
            (w_i \tz_i^{\otimes d}) \stimes (w_j \tz_j^{\otimes d}) \mid i \le j, i, j \in [s]
        \} )) = \{ 
            V_d(Z) C V_d(Z)\tran \mid C \in \sym{s}
         \} = V_d(Z) \sym{r} V_d(Z)\tran
    $.
    Similarly,
    $
        \psi(\linspan ( \{ (w_i \tz_i^{\otimes d}) \stimes (w_i \tz_i^{\otimes d}) \mid i \in [s] \} ))
        = \linspan (\{m_d(z_i) m_d(z_i)\tran\}_{i=1}^s).
    $
    These results, together with~\eqref{eq:sim:minface-proof-leq}, \eqref{eq:sim:minface-proof-nonzero}, Lemma~\ref{lem:sim:generic-conjuction} and $(\reals \times \Real{n})^{\times s}$'s irreducibility under Zariski topology, yield that
    \begin{align}
        \label{eq:sim:minface-proof-symsubspace}
        V_d(Z) \sym{s} V_d(Z)\tran \cap \Hankel{n+1, d} = \linspan (\{m_d(z_i) m_d(z_i)\tran\}_{i=1}^s)
    \end{align}
    holds for generic $\{w_i, z_i\}_{i=1}^s \in (\reals \times \Real{n})^{\times s}$ as long as $s$ satisfies~\eqref{eq:sim:minface-proof-s}.

    \textbf{Step 5.} 
    Next we show that for generically chosen $\{w_i, z_i\}_{i=1}^s \in (\reals \times \Real{n})^{\times s}$,
    \begin{align}
        \label{eq:sim:minface-proof-VdZ}
        \rank{V_d(Z)} = \rank{[m_d(z_1), \dots, m_d(z_s)]} = s
    \end{align}
    holds as long as $s \le \binom{n+d}{d}$. Since $V_d(Z)$ depends only on $Z := [z_1, \dots, z_s]$, it suffices to consider generic $Z \in \Real{n \times s}$. The set $\{Z \in \Real{n \times s} \mid \rank{V_d(Z)} < s\}$ is a variety, because $\rank{V_d(Z)} < s$ iff all $s \times s$ minors of $V_d(Z)$ vanish. Thus it remains to show that this variety is proper. Suppose $\linspan(\{m_d(z) \mid z \in \Real{n}\}) \subsetneq \Real{\binom{n+d}{d}}$. Then there exists $c \ne 0 \in \Real{\binom{n+d}{d}}$ such that $p(z) := c\tran m_d(z) = 0$ for all $z \in \Real{n}$. Writing $p(z) = \sum_{i=0}^d q_i(z_1, \dots, z_{n-1}) z_n^i$, we see that $q_i(z_1, \dots, z_{n-1}) = 0$ for all $i = 0, \dots, d$ and all $(z_1, \dots, z_{n-1}) \in \Real{n-1}$, so each $q_i$ is identically zero. Hence $p$ is the zero polynomial, which implies $c = 0$, a contradiction. Therefore, $\linspan(\{m_d(z) \mid z \in \Real{n}\}) = \Real{\binom{n+d}{d}}$, so there exist $\binom{n+d}{d}$ points whose moment vectors span $\Real{\binom{n+d}{d}}$. In particular, for every $s \le \binom{n+d}{d}$, there exists $Z$ such that $\rank{V_d(Z)} = s$. Hence $\{Z \in \Real{n \times s} \mid \rank{V_d(Z)} < s\}$ is a proper variety, and~\eqref{eq:sim:minface-proof-VdZ} holds generically.

    Choose $s$ satisfying~\eqref{eq:sim:minface-proof-s} so that~\eqref{eq:sim:minface-proof-symsubspace} holds. Then $s \le \binom{n+d}{d}$, so condition~\eqref{eq:sim:minface-proof-VdZ} is automatically satisfied. By Lemma~\ref{lem:sim:generic-conjuction}, for generically chosen $\{w_i, z_i\}_{i=1}^s \in (\reals \times \Real{n})^{\times s}$, conditions~\eqref{eq:sim:minface-proof-symsubspace},~\eqref{eq:sim:minface-proof-VdZ}, and~\eqref{eq:sim:minface-proof-nonzero} hold simultaneously.
    Intersect both sides of~\eqref{eq:sim:minface-proof-symsubspace} with $\psd{\binom{n+d}{d}}$. For the left-hand side, since $V_d(Z)$ has full column rank by~\eqref{eq:sim:minface-proof-VdZ}, Lemma~\ref{lem:sim:V-sym-Vt-psd} gives $V_d(Z) \sym{s} V_d(Z)\tran \cap \psd{\binom{n+d}{d}} = V_d(Z) \psd{s} V_d(Z)\tran$. Moreover,
    $
        \sum_{i=1}^s w_i^2 m_d(z_i) m_d(z_i)\tran = V_d(Z) \diag{\{w_i^2\}_{i=1}^s} V_d(Z)\tran
    $
    with $\diag{\{w_i^2\}_{i=1}^s} \in \pd{s}$ by~\eqref{eq:sim:minface-proof-nonzero}. Hence Lemma~\ref{lem:sim:minimal-faces-psd-cone} implies
    \begin{align}
        \label{eq:sim:minface-proof-LHS}
        V_d(Z) \sym{s} V_d(Z)\tran \cap \Hankel{n+1, d} \cap \psd{\binom{n+d}{d}} 
        = \minface\left( \sum_{i=1}^s w_i^2 m_d(z_i) m_d(z_i)\tran, \psd{\binom{n+d}{d}} \right) \cap \Hankel{n+1, d}.
    \end{align}

    For the right-hand side of \eqref{eq:sim:minface-proof-symsubspace}, since $m_d(z_i) m_d(z_i)\tran \in \psd{\binom{n+d}{d}}$ for all $i \in [s]$, we have
    $
        \cone (\{m_d(z_i) m_d(z_i)\tran\}_{i=1}^s) \subseteq \linspan (\{m_d(z_i) m_d(z_i)\tran\}_{i=1}^s) \cap \psd{\binom{n+d}{d}}.
    $
    Conversely, suppose $X = \sum_{i=1}^s c_i m_d(z_i) m_d(z_i)\tran = V_d(Z) \diag{\{c_i\}_{i=1}^s} V_d(Z)\tran$ lies in $\psd{\binom{n+d}{d}}$. Since $V_d(Z)$ has full column rank generically, for any $y \in \Real{s}$ there exists $x \in \Real{\binom{n+d}{d}}$ such that $y = V_d(Z)\tran x$. Therefore,
    $
        y\tran \diag{\{c_i\}_{i=1}^s} y
        = x\tran V_d(Z) \diag{\{c_i\}_{i=1}^s} V_d(Z)\tran x
        = x\tran X x \ge 0,
    $
    and hence $\diag{\{c_i\}_{i=1}^s} \in \psd{s}$. This implies $c_i \ge 0$ for all $i \in [s]$. Thus, generically,
    \begin{align}
        \label{eq:sim:minface-proof-RHS}
        \cone (\{m_d(z_i) m_d(z_i)\tran\}_{i=1}^s) = \linspan (\{m_d(z_i) m_d(z_i)\tran\}_{i=1}^s) \cap \psd{\binom{n+d}{d}}.
    \end{align}

    Combining~\eqref{eq:sim:minface-proof-symsubspace},~\eqref{eq:sim:minface-proof-LHS}, and~\eqref{eq:sim:minface-proof-RHS}, we obtain
    \begin{align*}
        \minface\left( \sum_{i=1}^s w_i^2 m_d(z_i) m_d(z_i)\tran, \psd{\binom{n+d}{d}} \right) \cap \Hankel{n+1, d}
        = \cone \left(\left\{m_d(z_i) m_d(z_i)\tran\right\}_{i=1}^s\right).
    \end{align*}
    Since the left-hand side is a face of $\MomConePs{n+1,2d} = \psd{\binom{n+d}{d}} \cap \Hankel{n+1,d}$ containing $\sum_{i=1}^s w_i^2 m_d(z_i) m_d(z_i)\tran$, and this matrix lies in its relative interior because all coefficients $w_i^2$ are strictly positive, it follows that
    \begin{align*}
        \minface\left( \sum_{i=1}^s w_i^2 m_d(z_i) m_d(z_i)\tran, \MomConePs{n+1,2d} \right)
        = \cone \left(\left\{m_d(z_i) m_d(z_i)\tran\right\}_{i=1}^s\right).
    \end{align*}
    This proves~\eqref{eq:sim:minface-condition}. The linear independence of $\{m_d(z_i) m_d(z_i)\tran\}_{i=1}^s$ comes from the linear independence of $\{m_d(z_i)\}_{i=1}^s$ in~\eqref{eq:sim:minface-proof-VdZ}. 

    It remains to show that the upper bound of $s$ in~\eqref{eq:sim:minface-s} grows as $\Theta(n^d)$ for fixed $d \ge 2$. When $d = 1$, it is easy to verify that the upper bound is $0$. When $d \ge 2$, however, since
    $
        \binom{n+d}{d} = \frac{1}{d!} \prod_{k=1}^d (n+k) = \frac{n^d}{d!} + O(n^{d-1})
    $
    and
    $
        \binom{n+d}{d}^{-1} \binom{n+2d}{2d}
        = \frac{d!}{(2d)!} \prod_{k=1}^d (n+d+k)
        = \frac{d!}{(2d)!} n^d + O(n^{d-1}),
    $
    the right-hand-side of~\eqref{eq:sim:minface-s} is 
    $
        \left( \frac{1}{2d!} - \frac{d!}{(2d)!} \right) n^d + O(n^{d-1}).
    $
    Moreover,
    $
        \frac{1}{2d!} - \frac{d!}{(2d)!}
        = \frac{(2d)! - 2(d!)^2}{2d!(2d)!} > 0,
    $
    because
    $
        \binom{2d}{d} = \frac{(2d)!}{(d!)^2} > 2
    $
    for all $d \ge 2$. Therefore, for fixed $d \ge 2$, the upper bound in~\eqref{eq:sim:minface-s} indeed grows as $\Theta(n^d)$.
\end{proof}

Theorem~\ref{thm:sim:minface-condition} should be viewed as a local facial statement rather than a global description of \( \MomConePs{n+1,2d} \). For fixed \( d \ge 2 \) and \( s \) in the regime~\eqref{eq:sim:minface-s}, it shows that a moment matrix formed by \( s \) generically chosen weighted atoms lies on a simplicial minimal face generated by the planted rank-one atomic rays, so that the cone locally behaves like \( \mathbb{R}_+^s \). In this sense, although the pseudo-moment cone may be globally complicated, it \emph{regularizes} around such matrices to a cone with transparent facial structure. This does not contradict the known existence of higher-rank extreme rays in \( \SosCone{n+1,2d}^* \), since those results concern the ambient cone globally, whereas Theorem~\ref{thm:sim:minface-condition} concerns only the minimal face containing moment matrices of this particular form. This local simplicial structure is exactly the geometric mechanism that later enables the extreme-ray decomposition procedure to recover the planted atoms and weights uniquely up to reordering, as we will discuss in \S\ref{sec:ray:unique}.


\section{Carath\'eodory-Type Atomic Decomposition of Moment Matrices}
\label{sec:ray}

In this section, we study the properties of Algorithm~\ref{alg:intro:extremeray-decomp} in detail. In \S\ref{sec:ray:exact}, we prove that, under exact arithmetic, Algorithm~\ref{alg:intro:extremeray-decomp} almost surely produces an extreme-ray decomposition for any given pair $(\calC, X)$. In \S\ref{sec:ray:unique}, we further show that, when $\calC = \MomConePs{n+1,2d}$ and $X = \sum_{i=1}^s w_i^2 m_d(z_i) m_d(z_i)\tran$, Algorithm~\ref{alg:intro:extremeray-decomp} can almost surely recover, up to a proper reordering, the atoms $\{z_i\}_{i=1}^s$ and weights $\{w_i^2\}_{i=1}^s$, provided that the conditions of Theorem~\ref{thm:sim:minface-condition} hold. 
In \S\ref{sec:ray:finite}, we introduce several techniques to stabilize Algorithm~\ref{alg:intro:extremeray-decomp} under finite-precision arithmetic, leading to the practical Algorithm~\ref{alg:ray:extremeray-decomp-robust}. Finally, in \S\ref{sec:ray:complexity}, we estimate the time and memory complexity of Algorithm~\ref{alg:ray:extremeray-decomp-robust} for the case $\calC = \MomConePs{n+1,2d}$. 

\subsection{Correctness of Algorithm~\ref{alg:intro:extremeray-decomp} under Exact Arithmetic}
\label{sec:ray:exact}

We first provide some elementary lemmas on convex geometry. 
\begin{lemma}[Random linear functional exposes an extreme point]
    \label{lem:ray:extreme-point}
    Let $\calK \subseteq \Real{N}$ be a nonempty compact convex set. For each $c \in \Real{N}$, define $\calE(c) := \argmin_{x \in \calK} \inprod{c}{x}$. Then there exists a Lebesgue-null set $\calN \subseteq \Real{N}$ such that, for every $c \in \Real{N} \setminus \calN$, the set $\calE(c)$ is a singleton, and its unique element is an extreme point of $\calK$.
\end{lemma}
\begin{proof}
    Define the support function $h_{\calK}(u) := \max_{x \in \calK} \inprod{u}{x}$ and the supporting face $\calF(\calK, u) := \{x \in \calK \mid \inprod{u}{x} = h_{\calK}(u)\}$, so that $\calE(c) = \calF(\calK, -c)$ for all $c \in \Real{N}$. Since $\calK$ is compact, there exists $R > 0$ such that $\normtwo{x} \le R$ for all $x \in \calK$, and hence $\abs{h_{\calK}(u) - h_{\calK}(v)} \le \max_{x \in \calK} \abs{\inprod{u-v}{x}} \le R \normtwo{u-v}$ for all $u, v \in \Real{N}$. Thus $h_{\calK}$ is Lipschitz, so by Rademacher's theorem it is differentiable almost everywhere. By~\cite[Theorem 1.7.4 and Corollary 1.7.3]{schneider13book-convex-bodies}, there exists a Lebesgue-null set $\calN' \subseteq \Real{N}$ such that, for every $u \in \Real{N} \setminus \calN'$, the face $\calF(\calK, u)$ is a singleton, say $\calF(\calK, u) = \{x_\star(u)\}$. We claim that $x_\star(u)$ is an extreme point of $\calK$. Indeed, if $x_\star(u) = \frac{1}{2}(y+z)$ for some $y, z \in \calK$, then maximality of $x_\star(u)$ gives $\inprod{u}{x_\star(u)} \ge \inprod{u}{y}$ and $\inprod{u}{x_\star(u)} \ge \inprod{u}{z}$, while linearity gives $\inprod{u}{x_\star(u)} = \frac{1}{2}(\inprod{u}{y} + \inprod{u}{z})$. Hence $\inprod{u}{y} = \inprod{u}{z} = h_{\calK}(u)$, so $y, z \in \calF(\calK, u)$, and therefore $y = z = x_\star(u)$. Finally, letting $\calN := -\calN'$, which is again Lebesgue-null, we obtain that for every $c \in \Real{N} \setminus \calN$, the set $\calE(c) = \calF(\calK, -c)$ is a singleton whose unique element is an extreme point of $\calK$.
\end{proof}

\begin{lemma}[Faces of a spectrahedral cone]
    \label{lem:ray:face}
    Let $\calC = \calL \cap \psd{N}$ be a nonempty spectrahedral cone, where $\calL \subseteq \sym{N}$ is a linear subspace. Then every nonempty face $\calF$ of $\calC$ is again a spectrahedral cone.
\end{lemma}
\begin{proof}
    Pick $X \in \ri(\calF)$. Since $\calF$ is the minimal face of $\calC$ containing $X$, we have $\calF = \minface(X,\calC) = \minface(X,\psd{N}) \cap \calL$ by~\cite[Theorem 3.3.1(2)(b)]{wolkowicz12book-handbook}. Moreover,  $\minface(X,\psd{N}) = Q\psd{r}Q\tran$ for $Q \in \Real{N \times r}$ with orthogonal columns and some $r$. Hence $\calF = \calL \cap Q\psd{r}Q\tran$, which is again a spectrahedral cone.
\end{proof}

\begin{lemma}[Extreme points and extreme rays]
    \label{lem:ray:point-ray}
    Let $\calC \subseteq \psd{N}$ be a nonempty spectrahedral cone, and define $\calB := \calC \cap \{Y \in \sym{N} \mid \trace{Y} = 1\}$. Then $X \in \calB$ is an extreme point of $\calB$ iff $\ray X$ is an extreme ray of $\calC$.
\end{lemma}
\begin{proof}
    Since every nonzero element of $\calC$ is PSD, it has positive trace. Thus every nonzero ray of $\calC$ intersects $\calB$ in exactly one point. Suppose first that $X$ is an extreme point of $\calB$, and let $X = Y + Z$ with $Y, Z \in \calC$. If both $Y$ and $Z$ are nonzero, then $1 = \trace{X} = \trace{Y} + \trace{Z}$, and so $X = \trace{Y}\frac{Y}{\trace{Y}} + \trace{Z}\frac{Z}{\trace{Z}}$, where $\frac{Y}{\trace{Y}}, \frac{Z}{\trace{Z}} \in \calB$. Since $X$ is an extreme point of $\calB$, both normalized matrices must equal $X$. Hence $Y, Z \in \ray X$, and therefore $\ray X$ is an extreme ray of $\calC$. Conversely, suppose that $\ray X$ is an extreme ray of $\calC$, and let $X = \lambda Y + (1-\lambda)Z$ with $Y, Z \in \calB$ and $0 < \lambda < 1$. Then $\lambda Y, (1-\lambda)Z \in \calC$. By extremality of $\ray X$, we have $Y, Z \in \ray X$. Since $\trace{Y} = \trace{Z} = \trace{X} = 1$, it follows that $Y = Z = X$. Thus $X$ is an extreme point of $\calB$.
\end{proof}

Now we are ready to show that, for a general spectrahedral cone $\calC$, Algorithm~\ref{alg:intro:extremeray-decomp} produces, with probability $1$, an extreme-ray decomposition of any nonzero $X \in \calC$. 
\begin{theorem}[Algorithm~\ref{alg:intro:extremeray-decomp} generates extreme rays]
    \label{thm:ray:correctness}
    Let $\calC \subseteq \sym{N}$ be a nonempty spectrahedral cone, and let $X \in \calC$ be nonzero. Suppose that, at every iteration,~\eqref{eq:intro:Mk} and~\eqref{eq:intro:tk} can be solved exactly. Then every step of Algorithm~\ref{alg:intro:extremeray-decomp} is well defined, and Algorithm~\ref{alg:intro:extremeray-decomp} terminates after finitely many steps. Moreover, with probability $1$, its output $\{t_k, M_k\}_{k=1}^r$ is an extreme ray decomposition of $X$ in $\calC$, \ie $X = \sum_{k=1}^r t_k M_k$, where $t_k > 0$ and $\ray M_k$ is an extreme ray of $\calC$ for every $k \in [r]$.
\end{theorem}
\begin{proof}
    (\romannumeral1) We first show that in each step $k$,~\eqref{eq:intro:Mk} is well-defined and $\ray M_k$ is an extreme ray of $\calC$ almost surely. Since $\calC$ is a nonempty spectrahedral cone, the minimal face $\calF_k$ defined in step $k$ (Algorithm~\ref{alg:intro:extremeray-decomp}, Line~\ref{alg:intro:line:Fk}) is a nonempty face of $\calC$. Hence $\calF_k$ is also a spectrahedral cone by Lemma~\ref{lem:ray:face}, and in particular $\calF_k \subseteq \calC \subseteq \psd{N}$. Therefore $\calG_k := \calF_k \cap \{Y \in \sym{N} \mid \trace{Y} = 1\}$ is a nonempty compact convex set. Indeed, since $X_k \in \calF_k \setminus \{0\}$ and $X_k \in \psd{N}$, we have $\trace{X_k} > 0$, so $X_k/\trace{X_k} \in \calG_k$. Thus the minimum in~\eqref{eq:intro:Mk} is attained. 
    
    Since $B_k$ in step $k$ (Algorithm~\ref{alg:intro:extremeray-decomp}, Line~\ref{alg:intro:line:Bk}) is drawn from an absolutely continuous distribution with respect to Lebesgue measure on $\sym{N}$, Lemma~\ref{lem:ray:extreme-point} implies that minimizing the random functional $f(Y) := \inprod{B_k}{Y}$ over $\calG_k$ almost surely returns an extreme point of $\calG_k$. By Lemma~\ref{lem:ray:point-ray}, $M_k$ from~\eqref{eq:intro:Mk} therefore determines an extreme ray $\ray M_k$ of $\calF_k$ almost surely. Now suppose $Y, Z \in \calC$ satisfy $Y + Z = M_k$. Since $M_k \in \calF_k$ and $\calF_k$ is a face of $\calC$, it follows that $Y, Z \in \calF_k$. By extremality of $\ray M_k$ in $\calF_k$, we obtain $Y, Z \in \ray M_k$. Hence $\ray M_k$ is also an extreme ray of $\calC$.

    (\romannumeral2) We next show that $t_k > 0$ in~\eqref{eq:intro:tk} exists and is unique. Define $I_k := \{t \ge 0 \mid X_k - t M_k \in \calF_k\}$. Since $X_k \in \ri(\calF_k)$ and $M_k \in \calF_k$, there exists $\epsilon > 0$ such that $X_k - t M_k \in \calF_k$ for all $t \in [0, \epsilon]$. Hence $[0, \epsilon] \subseteq I_k$. Since the map $t \mapsto X_k - t M_k$ is continuous and $\calF_k$ is closed, the set $I_k$ is closed. Since this map is affine and $\calF_k$ is convex, the set $I_k$ is convex in $\reals$. Moreover, if $t \in I_k$, then $X_k - t M_k \in \calF_k \subseteq \psd{N}$, so $0 \le \trace{X_k - t M_k} = \trace{X_k} - t$, because $\trace{M_k} = 1$. Thus $t \le \trace{X_k}$. Therefore $I_k$ is a nonempty closed convex subset of $\reals$ that is bounded above, so $I_k = [0, t_k]$ for some unique $t_k > 0$.

    (\romannumeral3) Finally, we show finite termination. By construction, $X_{k+1} = X_k - t_k M_k \in \calF_k$. If $X_{k+1} \ne 0$, then $X_{k+1}$ cannot lie in $\ri(\calF_k)$. Otherwise, since $M_k \in \calF_k$, there would exist $\delta > 0$ such that $X_{k+1} - \delta M_k \in \calF_k$. This would imply $X_k - (t_k + \delta)M_k \in \calF_k$, contradicting the definition of $t_k$. Hence $X_{k+1}$ lies in the relative boundary of $\calF_k$ whenever $X_{k+1} \ne 0$. Therefore $\calF_{k+1} = \minface(X_{k+1}, \calC) \subsetneq \calF_k$, and so $\dim(\calF_{k+1}) < \dim(\calF_k)$. Since face dimensions are nonnegative integers, the iteration sequence terminates after finitely many steps. 
    
    Suppose the total iteration number is $r$. For each $k \in [r]$, we obtain a weight $t_k > 0$ and a matrix $M_k$ such that $\ray M_k$ is an extreme ray of $\calC$ almost surely. Summing the identities $X_{k+1} = X_k - t_k M_k$ from $k = 1$ to $r$ gives $X = \sum_{k=1}^r t_k M_k$. Since the algorithm terminates after finitely many steps, the intersection of the corresponding finitely many probability-$1$ events still has probability $1$. Therefore, with probability $1$, the output $\{t_k, M_k\}_{k=1}^r$ forms an extreme ray decomposition of $X$, \ie $X = \sum_{k=1}^r t_k M_k$ with $t_k > 0$ and $\ray M_k$ an extreme ray of $\calC$ for all $k \in [r]$.
\end{proof}

\subsection{Simplicial Regularizability Enforces Uniqueness of Decomposition}
\label{sec:ray:unique}

As one may have already noticed, in~\eqref{eq:intro:Mk}, different choices of the linear functional $B_k$ may produce different extreme rays $M_k$ even when $\calF_k$ is fixed. In other words, although Algorithm~\ref{alg:intro:extremeray-decomp} almost surely returns an extreme ray decomposition of $X \in \calC$, the decomposition need not be unique. For the same input $X$, even the number of iterations $r$ may vary, let alone the collection $\{t_k, M_k\}_{k=1}^r$. For example, if $\calC = \psd{N}$ and $X \in \pd{N}$, then $\minface(X,\calC) = \psd{N}$, and minimizing a random linear functional may attain any rank-$1$ extreme ray of $\psd{N}$, \ie $\ray(xx\tran)$ for arbitrary nonzero $x \in \reals^n$.

In the special case where $X = \sum_{i=1}^s w_i^2 m_d(z_i)m_d(z_i)\tran$ and $\calC = \MomConePs{n+1,2d}$, however, the situation is markedly different. Indeed, as long as $s$ satisfies~\eqref{eq:sim:minface-s} and $\{w_i, z_i\}_{i=1}^s$ is chosen generically from $(\reals \times \Real{n})^s$, the minimal face $\minface(X,\calC)$ becomes a simplicial cone from Theorem~\ref{thm:sim:minface-condition}. The following lemma shows that, with probability $1$, minimizing a random linear functional over a simplicial cone can only select its linearly independent generators.
\begin{lemma}[Random linear functional selects a generator of a simplicial cone]
    \label{lem:ray:simplicial-hit}
    Let $\{ H_i \}_{i=1}^s \in \psd{N}$ be linearly independent. Define
    $\calC := \cone(\{ H_i \}_{i=1}^s)$ and $\calG := \calC \cap \{X \in \sym{N} \mid \trace{X} = 1\}$.
    Let $B \in \sym{N}$ be drawn from a distribution that is absolutely continuous with respect to Lebesgue measure on $\sym{N}$. Then, with probability $1$,
    $\argmin_{Y \in \calG} \inprod{B}{Y}$
    is a singleton of the form $\{H_i/\trace{H_i}\}$ for some $i \in [s]$.
\end{lemma}
\begin{proof}
    Since $\{ H_i \}_{i=1}^s$ are linearly independent, each $H_i$ is nonzero. As $H_i \in \psd{N}$, this implies $\trace{H_i} > 0$ for all $i \in [s]$. Now let $Y \in \calG$. Then $Y \in \cone(\{ H_i \}_{i=1}^s)$, so $Y = \sum_{i=1}^s \lambda_i H_i$ for some $\lambda_i \ge 0$. Since $\trace{Y} = 1$, we also have $\sum_{i=1}^s \lambda_i \trace{H_i} = 1$. Writing $\mu_i := \lambda_i \trace{H_i}$, we obtain $\mu_i \ge 0$, $\sum_{i=1}^s \mu_i = 1$, and
    $Y = \sum_{i=1}^s \mu_i \frac{H_i}{\trace{H_i}}$.
    Hence $\calG = \conv(\{H_i/\trace{H_i}\}_{i=1}^s)$. Since positive rescaling preserves linear independence, the matrices $\{H_i/\trace{H_i}\}_{i=1}^s$ are linearly independent, and therefore they are precisely the extreme points of $\calG$. The conclusion now follows from Lemma~\ref{lem:ray:extreme-point}.
\end{proof}

Now we combine Theorem~\ref{thm:sim:minface-condition} and Theorem~\ref{thm:ray:correctness} together.
\begin{theorem}[Atomic decomposition via extreme ray decomposition]
    \label{thm:ray:atomic-decomp}
    Fix $d \ge 2$. Given $n \ge 1$, let $\MomConePs{n+1,2d}$ be defined in~\eqref{eq:pre:moment-cone-ps}. Suppose $s = O(n^d)$ satisfies~\eqref{eq:sim:minface-s}. For generically chosen $\{(w_i, z_i)\}_{i=1}^s \in (\reals \times \Real{n})^{\times s}$, let $X = \sum_{i=1}^s w_i^2 m_d(z_i)m_d(z_i)\tran$ and $\calC = \MomConePs{n+1,2d}$. If~\eqref{eq:intro:Mk} and~\eqref{eq:intro:tk} can be solved in exact arithmetic, then Algorithm~\ref{alg:intro:extremeray-decomp} almost surely returns $\{t_k, M_k\}_{k=1}^r$ with (up to a proper reordering)
    \begin{align}
        \label{eq:ray:atomic-tk-Mk}
        r = s, \quad \text{and} \quad t_k = w_k^2 \normtwo{m_d(z_k)}^2, \quad M_k = \frac{1}{\normtwo{m_d(z_k)}^2} m_d(z_k) m_d(z_k)\tran, \ \forall k \in [s]. 
    \end{align}
\end{theorem}
\begin{proof}
    For each $k \in [s]$, since $s-k+1 \le s$, condition~\eqref{eq:sim:minface-s} also holds with $s$ replaced by $s-k+1$. Let
    $
        \pi_k : (\reals \times \Real{n})^{\times s} \to (\reals \times \Real{n})^{\times (s-k+1)}
    $
    be the projection onto the last $s-k+1$ coordinates. Define a property $P_k$ on $(\reals \times \Real{n})^{\times (s-k+1)}$ by declaring that $P_k(\{(\widehat w_i,\widehat z_i)\}_{i=k}^s)$ holds if
    \begin{align*}
        \minface\left(\sum_{i=k}^s \widehat w_i^2 m_d(\widehat z_i)m_d(\widehat z_i)\tran, \MomConePs{n+1,2d}\right)
        = \cone\left(\left\{m_d(\widehat z_i)m_d(\widehat z_i)\tran\right\}_{i=k}^s\right),  
    \end{align*}
    and $\{m_d(\widehat z_i)m_d(\widehat z_i)\tran\}_{i=k}^s$ are linearly independent. By Theorem~\ref{thm:sim:minface-condition}, $P_k$ holds generically on $(\reals \times \Real{n})^{\times (s-k+1)}$. Since $\pi_k$ is a polynomial map and
    $
    \pi_k((\reals \times \Real{n})^{\times s}) = (\reals \times \Real{n})^{\times (s-k+1)},
    $
    Lemma~\ref{lem:sim:generic-polymap} implies that $P_k \circ \pi_k$ holds generically on $(\reals \times \Real{n})^{\times s}$. Since this is true for every $k \in [s]$, Lemma~\ref{lem:sim:generic-conjuction} shows that, for generically chosen $\{(w_i,z_i)\}_{i=1}^s$, the conclusion of Theorem~\ref{thm:sim:minface-condition} holds simultaneously for every tail tuple $\{(w_i,z_i)\}_{i=k}^s$, $k \in [s]$.

    At the first iteration of Algorithm~\ref{alg:intro:extremeray-decomp}, by Theorem~\ref{thm:sim:minface-condition},  $\calF_1 = \minface(X_1, \calC) =$ $\cone(\{m_d(z_i) m_d(z_i)\tran\}_{i=1}^s)$, with $\{m_d(z_i) m_d(z_i)\tran\}_{i=1}^s$ linearly independent and PSD. Invoking Lemma~\ref{lem:ray:simplicial-hit}, with probability $1$,
    $
        M_1 = \frac{1}{\trace{m_d(z_i) m_d(z_i)\tran}} m_d(z_i) m_d(z_i)\tran
        = \frac{1}{\normtwo{m_d(z_i)}^2} m_d(z_i) m_d(z_i)\tran
    $
    for some $i \in [s]$. After a proper reordering, we can choose $i = 1$.

    Next we calculate $t_1$ in~\eqref{eq:intro:tk}. Notice that $t$ must satisfy
    \begin{align*}
         X_1 - t M_1
        = & \sum_{i=2}^s w_i^2 m_d(z_i) m_d(z_i)\tran
        + \left( w_1^2 - \frac{t}{\normtwo{m_d(z_1)}^2} \right) m_d(z_1) m_d(z_1)\tran \\
        \in & \ \cone(\{m_d(z_i) m_d(z_i)\tran\}_{i=1}^s)
        = \left\{
            \sum_{i=1}^s c_i m_d(z_i) m_d(z_i)\tran \mymid c_i \ge 0, \forall i \in [s]
        \right\}.
    \end{align*}
    Hence by definition, $t_1 = w_1^2 \normtwo{m_d(z_1)}^2$. And $X_2 = \sum_{i=2}^s w_i^2 m_d(z_i) m_d(z_i)\tran$.

    Now suppose, after a proper reordering, that for some $k \ge 1$, $X_k = \sum_{i=k}^s w_i^2 m_d(z_i) m_d(z_i)\tran$. By the first paragraph of the proof and Theorem~\ref{thm:sim:minface-condition} applied to the tail tuple $\{(w_i,z_i)\}_{i=k}^s$, we get
    $
        \calF_k = \minface(X_k, \calC) = \cone(\{m_d(z_i) m_d(z_i)\tran\}_{i=k}^s),
    $
    where $\{m_d(z_i) m_d(z_i)\tran\}_{i=k}^s$ are linearly independent and PSD. Applying Lemma~\ref{lem:ray:simplicial-hit} again, with probability $1$,
    $
        M_k = \frac{1}{\normtwo{m_d(z_k)}^2} m_d(z_k) m_d(z_k)\tran
    $
    after a proper reordering of the remaining indices. Then, exactly as in the case $k = 1$, the scalar $t_k$ in~\eqref{eq:intro:tk} must satisfy
    \begin{align*}
        X_k - t M_k
        = \sum_{i=k+1}^s w_i^2 m_d(z_i) m_d(z_i)\tran
        + \left( w_k^2 - \frac{t}{\normtwo{m_d(z_k)}^2} \right) m_d(z_k) m_d(z_k)\tran 
        \in \cone(\{m_d(z_i) m_d(z_i)\tran\}_{i=k}^s).
    \end{align*}
    Therefore, $t_k = w_k^2 \normtwo{m_d(z_k)}^2$, and consequently
    $
        X_{k+1} = X_k - t_k M_k = \sum_{i=k+1}^s w_i^2 m_d(z_i) m_d(z_i)\tran.
    $

    Continuing inductively yields $M_k = \frac{1}{\normtwo{m_d(z_k)}^2} m_d(z_k) m_d(z_k)\tran$ and $t_k = w_k^2 \normtwo{m_d(z_k)}^2$ for all $k \in [s]$, up to a proper reordering. Moreover, $X_{s+1} = 0$, so the algorithm stops no later than the $s$-th iteration. On the other hand, for each $k \in [s]$, $X_k = \sum_{i=k}^s w_i^2 m_d(z_i) m_d(z_i)\tran \ne 0$, since all coefficients are strictly positive and the summands are PSD rank-one matrices. Hence the algorithm does not stop before iteration $s$, and thus $r = s$.
    Finally, there are only finitely many iterations, and at each iteration Lemma~\ref{lem:ray:simplicial-hit} succeeds with probability $1$. Therefore the intersection of these finitely many probability-$1$ events still has probability $1$. Hence Algorithm~\ref{alg:intro:extremeray-decomp} almost surely returns $\{t_k,M_k\}_{k=1}^r$ satisfying (up to a proper reordering)~\eqref{eq:ray:atomic-tk-Mk}.
\end{proof}

Theorem~\ref{thm:ray:atomic-decomp} yields a new atomic decomposition algorithm for moment matrices of the form $X = \sum_{i=1}^s w_i^2 m_d(z_i) m_d(z_i)\tran$. Specifically, when the weights and atoms are chosen generically and the number of atoms satisfies $s = O(n^d)$ for fixed $d \ge 2$, the recovery guarantee is given by the following corollary.
\begin{corollary}[Recovery of weights and atoms]
    \label{cor:ray:recover}
    Consider the setting of Theorem~\ref{thm:ray:atomic-decomp}. Suppose Algorithm~\ref{alg:intro:extremeray-decomp} returns $\{t_k, M_k\}_{k=1}^s$. Then, up to the same proper reordering as in Theorem~\ref{thm:ray:atomic-decomp}, for each $k \in [s]$ we can recover $w_k^2$ and $z_k$ as follows:
    \begin{enumerate}
        \item Compute the eigenvalue decomposition of $M_k$. Let $\lambda_k$ be its unique nonzero eigenvalue, and let $h_k$ be a corresponding unit eigenvector.
        \item Recover $m_d(z_k)$ by
        $
            m_d(z_k) = \frac{1}{h_{k,1}} h_k.
        $
        Then recover $z_k$ from the degree-one entries of $m_d(z_k)$.
        \item Recover $w_k^2$ by
        $
            w_k^2 = t_k h_{k,1}^2.
        $
    \end{enumerate}
\end{corollary}
\begin{proof}
    By Theorem~\ref{thm:ray:atomic-decomp}, up to a proper reordering,
    $M_k = \frac{1}{\normtwo{m_d(z_k)}^2} m_d(z_k)m_d(z_k)\tran$ and $t_k = w_k^2 \normtwo{m_d(z_k)}^2$.
    Hence $M_k$ is rank-one PSD with unique nonzero eigenvalue $1$, and any corresponding unit eigenvector has the form $h_k = \pm \frac{m_d(z_k)}{\normtwo{m_d(z_k)}}$.
    Since the first entry of $m_d(z_k)$ equals $1$, we have $h_{k,1} = \pm \frac{1}{\normtwo{m_d(z_k)}} \ne 0$, 
    and therefore $\frac{1}{h_{k,1}} h_k = m_d(z_k)$. 
    Thus $z_k$ is recovered from the degree-one entries of $m_d(z_k)$. Finally,
    $
        w_k^2
        = \frac{t_k}{\normtwo{m_d(z_k)}^2}
        = t_k h_{k,1}^2.
    $
\end{proof}

We note that Algorithm~\ref{alg:intro:extremeray-decomp} and Corollary~\ref{cor:ray:recover} can recover asymptotically more atoms than direct flatness-based extraction from the given moment matrix~\cite{henrion05ppc-detecting,klep18siopt-minimizer-extraction-robust}. For fixed \( d \ge 2 \), the usual flatness test at degree \(2d\) can certify and extract at most \( \Theta(n^{d-1}) \) atoms, while providing an \emph{a priori} sufficient rank condition for existence and recovery. This limitation can be bypassed by hierarchy-based approaches that complete higher-order moments until a flat extension is found~\cite{nie14fcm-truncated-moment}. In contrast, our method decomposes the given degree-\(d\) moment matrix directly, using the simplicial regularizability of the pseudo-moment cone, and allows recovery of up to \( \Theta(n^d) \) atoms. The tradeoff is that our method does not currently come with an analogous \emph{a priori} sufficient certificate.


\subsection{Stabilizing Algorithm~\ref{alg:intro:extremeray-decomp} under Finite-Precision Arithmetic}
\label{sec:ray:finite}

Theorem~\ref{thm:ray:correctness}, Theorem~\ref{thm:ray:atomic-decomp}, and Corollary~\ref{cor:ray:recover} are all established under the assumption that~\eqref{eq:intro:Mk} and~\eqref{eq:intro:tk} can be solved in exact arithmetic. Unfortunately, this assumption is unrealistic in practice. In fact, a direct implementation of Algorithm~\ref{alg:intro:extremeray-decomp} typically becomes numerically unstable after only \(2\) or \(3\) iterations due to the accumulation of various numerical errors.

To address this issue, we develop several techniques to stabilize Algorithm~\ref{alg:intro:extremeray-decomp}, including (\romannumeral1) facial reduction to improve the stability of the SDP solved in~\eqref{eq:intro:Mk}; (\romannumeral2) a closed-form expression for \(t_k\) in~\eqref{eq:intro:tk}; (\romannumeral3) an alternating projection subroutine to improve feasibility; and (\romannumeral4) a restarting mechanism to automatically handle SDP failures and accumulated numerical error. Throughout this section, we represent the spectrahedral cone \(\calC \subseteq \sym{N}\) as \(\calC = \sym{N}_+ \cap \{ X \in \sym{N} \mid \calA(X) = 0 \}\), where the linear map \(\calA: \sym{N} \to \Real{m}\) is defined by \(\calA(X) = [\inprod{A_1}{X}, \dots, \inprod{A_m}{X}]^\top\) with \(A_i \in \sym{N}\) for each \(i \in [m]\). We denote by \(\vect(A_i)\) the vectorization of \(A_i\), and by \(\vect(\calA)\) the matrix \(\vect(\calA) := [\vect(A_1), \dots, \vect(A_m)]\). The resulting robust implementation is summarized in Algorithm~\ref{alg:ray:extremeray-decomp-robust}. We now describe these four stabilization techniques in detail.

\begin{algorithm}
    \caption{Robust Carath\'eodory-Type Extreme Ray Decomposition in Finite Precision \label{alg:ray:extremeray-decomp-robust}}
    \begin{algorithmic}[1]
        \Require linear map $\calA: \sym{N} \to \Real{m}$ representing a nonempty spectrahedral cone $\calC = \psd{N} \cap \{ X \mid \calA(X) = 0 \}$, initial point $X \ne 0 \in \calC$, tolerances $\epsrank, \epsbreak, \epscol, \epsalt$, iteration threshold $\Talt$
        \Ensure $\{t_k, M_k\}_{k=1}^r$, with $X = \sum_{k=1}^r t_k M_k$ and $\ray M_k$ numerically extreme in $\calC$

        \Procedure{RayDecompRestart}{$\calA, X, \epsrank, \epsbreak, \epscol, \epsalt, \Talt$}
            \State $\ell := 1$, $r := 0$, $X_1^{\mathrm{res}} := X$, $\epsrank^\circ := \epsrank$


            \State Let $\{\lambda_{i}\}_{i=1}^N$ be eigenvalues of $X$ and set $r_{\max} := \abs{\{ i \in [N] \mid \lambda_i > \epsrank^\circ \}}$ \Comment{compute numerical rank}
            \For{$\ell=1,\dots,r_{\max}$}
                \vspace{-2mm} \Statex \hdashrule{\linewidth}{0.4pt}{2pt 2pt}

                \State Compute $(s_\ell, \{\tau_{\ell,j}, N_{\ell,j}\}_{j=1}^{s_\ell}, X_{\ell+1}^{\mathrm{res}}) := \Call{RayDecompFR}{\calA, X_\ell^{\mathrm{res}}, \epsrank, \epsbreak, \epscol, \epsalt, \Talt}$
                \State Append $\{\tau_{\ell,j}, N_{\ell,j}\}_{j=1}^{s_\ell}$ to $\{t_k, M_k\}_{k=1}^r$, and set $r := r + s_\ell$

                \vspace{-2mm} \Statex \hdashrule{\linewidth}{0.4pt}{2pt 2pt}

                \State If $\normF{X_{\ell+1}^{\mathrm{res}}} < \epsbreak$, \Return $r$, $\{t_k, M_k\}_{k=1}^r$
                \Comment{succeed, early stop}
                \State Set $\epsrank := \epsrank/10$ if $s_\ell = 0$, and $\epsrank := \epsrank^\circ$ otherwise
                \Comment{adapt $\epsrank$}
                \label{alg:ray:line:adpt-epsrank}
                \State $X_{\ell+1}^{\mathrm{res}} := \Call{AlternatingProj}{X_{\ell+1}^{\mathrm{res}}, \calA, \Talt, \epsalt}$
                \Comment{\cf Algorithm~\ref{alg:ray:alternating-projection}}
                \label{alg:ray:line:reround}
                
                \vspace{-2mm} \Statex \hdashrule{\linewidth}{0.4pt}{2pt 2pt}
            \EndFor
            \State \Return $r$, $\{t_k, M_k\}_{k=1}^r$
        \EndProcedure

        \Statex

        \Procedure{RayDecompFR}{$\calA, X, \epsrank, \epsbreak, \epscol, \epsalt, \Talt$} \label{alg:ray:line:start-fr}
            \State $k := 1$, $X_1 := X$, $\calI_0 := [m]$
            \While{$\normF{X_k} \ge \epsbreak$}

                \vspace{-2mm} \Statex \hdashrule{\linewidth}{0.4pt}{2pt 2pt}

                \State Compute an eigenvalue decomposition $X_k = V_k D_k V_k\tran$ \Comment{compute numerical rank}
                \label{alg:ray:line:inner-start}

                \State Let $\lambda_{k,1} \ge \dots \ge \lambda_{k,N}$ denote $\diag{D_k}$. Set $r_k := \abs{\{ i \in [N] \mid \lambda_{k,i} > \epsrank \}}$, $Q_k := V_k(:,1:r_k)$

                \vspace{-2mm} \Statex \hdashrule{\linewidth}{0.4pt}{2pt 2pt}

                \State Set $\tX_k := Q_k\tran X_k Q_k$ \Comment{facial reduction} \label{alg:ray:line:facial-reduction-1}

                \State $\calJ_k := \Call{IndependentCols}{\vect(\{Q_k\tran A_i Q_k \}_{i \in \calI_{k-1}}),\epscol}$ \Comment{\cf Algorithm~\ref{alg:ray:remove-dependent-cols}}
                \label{alg:ray:line:facial-reduction-2}

                \State Set $\calI_k := \calI_{k-1}(\calJ_k)$ and define $\tcalA_k(Y) := (\inprod{Q_k\tran A_i Q_k}{Y})_{i \in \calI_k}$
                \State Let $\tcalF_k := \{ Y \in \sym{r_k}_+ \mid \widetilde{\mathcal{A}}_k(Y) = 0 \}$
                \State $B_k :=$ random matrix from any absolutely continuous distribution on $\sym{r_k}$
                \State Generate $\tM_k \in \sym{r_k}$ from the following linear SDP: $\tM_k \in \argmin_{Y \in \widetilde{\calF}_k, \trace{Y} = 1} \inprod{B_k}{Y}$
                \label{alg:ray:line:facial-reduction-3}

                \vspace{-2mm} \Statex \hdashrule{\linewidth}{0.4pt}{2pt 2pt}

                \State If SDP solver fails, 
                \Return $r := k-1$, $\{t_j, M_j\}_{j=1}^r$, $X_{r+1}$ \Comment{restart mechanism}

                \vspace{-2mm} \Statex \hdashrule{\linewidth}{0.4pt}{2pt 2pt}

                \State Lift $M_k := Q_k \tM_k Q_k^\top$

                \State Find $t_k := \lambda_{\max}^{-1}(\tX_k^{-\frac{1}{2}} \tM_k \tX_k^{-\frac{1}{2}})$
                \Comment{\cf~\eqref{eq:ray:closed-form-tk}}
                \label{alg:ray:line:closed-form-tk}

                \State $X_{k+1} := \Call{AlternatingProj}{X_k - t_k M_k, \mathcal{A}, \Talt, \epsalt} 
                $, and set $k := k+1$
                \Comment{\cf Algorithm~\ref{alg:ray:alternating-projection}}
                \label{alg:ray:line:inner-end}

                \vspace{-2mm} \Statex \hdashrule{\linewidth}{0.4pt}{2pt 2pt}
            \EndWhile
            \State \Return $r := k - 1$, $\{t_j, M_j\}_{j=1}^r$, $X_{r+1}$
        \EndProcedure
    \end{algorithmic}
\end{algorithm}

\begin{algorithm}
    \caption{Extraction of a Numerically Independent Column Subset~\label{alg:ray:remove-dependent-cols}}
    \begin{algorithmic}[1]
        \Require matrix $A \in \mathbb{R}^{m \times n}$, tolerance $\epscol > 0$
        \Ensure index set $\calI \subseteq [n]$ that makes the columns of $A(:, \calI)$ numerically linearly independent
        \Procedure{IndependentCols}{$A,\epscol$}
            \State Compute a column-pivoted QR factorization $A(:,p) = QR$
            \State Set $d := \abs{\diag{R}}$ and compute $r := \max \{ i \mid d_i > \epscol d_1 \}$
            \State \Return $\calI := p(1:r)$
        \EndProcedure
    \end{algorithmic}
\end{algorithm}

\begin{algorithm}
    \caption{Alternating Projection onto the Intersection of a Linear Subspace and the PSD Cone~\label{alg:ray:alternating-projection}}
    \begin{algorithmic}[1]
        \Require initial matrix $X$, linear operator $\calA: \sym{N} \to \Real{m}$ representing a subspace $\calL = \{X \mid \calA(X) = 0\} \subset \sym{N}$, iteration limit $\Talt$, tolerance $\epsalt$
        \Ensure projected matrix $X$
        \Procedure{AlternatingProj}{$X$, $\calA$, $\Talt$, $\epsalt$}
            \For{$k=1,\dots,\Talt$}
                \State Project onto the linear subspace: $X := X - \calA\tran (\calA \calA\tran)^{-1} \calA(X)$
                \label{alg:ray:line:alt-subspace}
                \State Project onto the PSD cone: $X := \Pi_{\psd{N}}(X)$
                \label{alg:ray:line:alt-psdcone}
                \State If $\normtwo{\calA(X)} < \epsalt$, \Return $X$
            \EndFor
            \State \Return $X$
        \EndProcedure
    \end{algorithmic}
\end{algorithm}

\paragraph{(\romannumeral1) Facial reduction.}
A major source of instability in Algorithm~\ref{alg:intro:extremeray-decomp} is that, when $\rank{X_k} < N$, solving~\eqref{eq:intro:Mk} directly in the ambient space $\sym{N}$ ignores that the feasible set is confined to a proper face of $\psd{N}$. In this case, strict feasibility in the ambient cone may fail, along with the regularity properties typically exploited by SDP solvers, including strong duality~\cite{ramana97siopt-strong-duality-sdp}. The effect can be particularly severe when $\rank{X_k} \ll N$. To address this issue, we apply facial reduction to restore strict feasibility on the relevant face~\cite{borwein81jams-facial-reduction}; see Lines~\ref{alg:ray:line:facial-reduction-1}--\ref{alg:ray:line:facial-reduction-3} of Algorithm~\ref{alg:ray:extremeray-decomp-robust}. Specifically, we compute the numerical rank $r_k$ of $X_k$ and an orthonormal basis $Q_k$ for its column space, both with tolerance $\epsrank$, and restrict the subproblem to the face $Q_k \psd{r_k} Q_k\tran$. We then remove numerically redundant constraints, up to tolerance $\epscol$, using a pivoted QR factorization (\cf Algorithm~\ref{alg:ray:remove-dependent-cols}). The resulting regularized linear SDP has size $r_k \times r_k$ and only $\abs{\calI_k}$ independent constraints. Since its feasible set satisfies Slater's condition and is compact, the primal and dual optimal values coincide, and both optima are attained.

\paragraph{(\romannumeral2) Closed-form expression for $t_k$.}
The scalar $t_k$ in~\eqref{eq:intro:tk} admits a closed-form expression. Let
$X_k = Q_k \tX_k Q_k\tran$
be the eigenvalue decomposition of $X_k$ at the $k$-th iteration of Algorithm~\ref{alg:intro:extremeray-decomp}, where $Q_k \in \Real{N \times r_k}$ has orthonormal columns and $\tX_k \in \pd{r_k}$ is diagonal. Since $M_k \in \calF_k = \minface(X_k, \calC)$, we must also have
$M_k = Q_k \tM_k Q_k\tran$
for some $\tM_k \in \psd{r_k}$. Therefore,
\begin{align}
    \label{eq:ray:closed-form-tk}
    t_k
    =  & \sup \{ t \ge 0 \mid X_k - t M_k \in \calF_k \} 
    = \sup \{ t \ge 0 \mid \tX_k - t \tM_k \in \psd{r_k} \} \nonumber \\
    = & \sup \{ t \ge 0 \mid I_{r_k} - t \tX_k^{-\frac{1}{2}} \tM_k \tX_k^{-\frac{1}{2}} \in \psd{r_k} \} 
    = \lambda_{\max}^{-1}\bigl( \tX_k^{-\frac{1}{2}} \tM_k \tX_k^{-\frac{1}{2}} \bigr).
\end{align}
The second equality holds because both $X_k$ and $M_k$ satisfy the linear constraints defining $\calC$, so $X_k - tM_k \in \{ X \mid \calA(X) = 0 \}$ for every $t$; hence only the PSD constraint remains. The final equality is well defined because $\trace{M_k} = \tr(\tM_k) = 1$ and $\tM_k \in \psd{r_k}$, which implies $\tX_k^{-\frac{1}{2}} \tM_k \tX_k^{-\frac{1}{2}}$ is nonzero and positive semidefinite. We have therefore replaced~\eqref{eq:intro:tk} by~\eqref{eq:ray:closed-form-tk} in Algorithm~\ref{alg:ray:extremeray-decomp-robust}, Line~\ref{alg:ray:line:closed-form-tk}.  

\paragraph{(\romannumeral3) Alternating projection.}
In exact arithmetic, $X_k \in \calC$ at every iteration of Algorithm~\ref{alg:intro:extremeray-decomp}. In practice, however, the SDP subproblems are not solved to machine precision, so the iterates gradually drift away from $\calC$, which can destabilize subsequent subproblems. To mitigate this effect, we apply a cheap alternating projection step (\cf Algorithm~\ref{alg:ray:alternating-projection}) after each iteration to improve the feasibility of $X_k$ with respect to $\calC$. 
The alternating projection routine is terminated after at most $\Talt$ steps, or earlier if the stopping criterion with tolerance $\epsalt$ is satisfied.
Since $\calA$ is typically sparse, the projection onto $\{X \mid \calA(X) = 0\}$ in Line~\ref{alg:ray:line:alt-subspace} can be implemented efficiently in two stages: we first precompute a sparse Cholesky factorization of $\calA \calA\tran$ at the beginning of Algorithm~\ref{alg:ray:extremeray-decomp-robust}$,$ and then, at each projection step, solve two sparse triangular systems. In practice, the cost of alternating projection is usually negligible compared with the SDP solve. A more detailed complexity discussion is deferred to \S\ref{sec:ray:complexity}. 

\paragraph{(\romannumeral4) A restarting mechanism.}
To further improve robustness, we restart the extreme-ray decomposition whenever an SDP subproblem cannot be solved reliably. In practice, such failures are often caused by overly aggressive eigenvalue truncation in Line~\ref{alg:ray:line:facial-reduction-1} of Algorithm~\ref{alg:ray:extremeray-decomp-robust}, which may render the reduced SDP in Line~\ref{alg:ray:line:facial-reduction-3} infeasible. Accordingly, if the inner routine $\textsc{RayDecompFR}$ (\cf Algorithm~\ref{alg:ray:extremeray-decomp-robust}, Line~\ref{alg:ray:line:start-fr}) returns no extreme rays, \ie when the SDP solver fails at the first iteration, we reduce $\epsrank$ by a factor of $10$ (\cf Algorithm~\ref{alg:ray:extremeray-decomp-robust}, Line~\ref{alg:ray:line:adpt-epsrank}). More generally, whenever an SDP failure is detected, we first project the current iterate back onto $\calC$ (\cf Algorithm~\ref{alg:ray:extremeray-decomp-robust}, Line~\ref{alg:ray:line:reround}) using alternating projection (\cf Algorithm~\ref{alg:ray:alternating-projection}), and then restart the procedure.

\subsection{Estimation of Complexity}
\label{sec:ray:complexity}

Now we estimate the time and memory complexity of Algorithm~\ref{alg:ray:extremeray-decomp-robust}. We begin with a general spectrahedral cone $\calC$ represented by $\calA: \sym{N} \to \Real{m}$, and then specialize to $\MomConePs{n+1,2d}$.

Since each generated $M_k$ decreases the iterate's rank by at least $1$, it is reasonable, by amortized analysis, to regard the total number of inner-loop iterations in $\textsc{RayDecompFR}$ (\cf Algorithm~\ref{alg:ray:extremeray-decomp-robust}, Line~\ref{alg:ray:line:inner-start} to~\ref{alg:ray:line:inner-end}) as $O(r_\max)$, where $r_\max \le N$ is the numerical rank of the input matrix $X$. Within each inner-loop iteration, the four main sources of complexity are: (\romannumeral1) eigenvalue decomposition to compute numerical rank (\cf Algorithm~\ref{alg:ray:extremeray-decomp-robust}, Line~\ref{alg:ray:line:inner-start}); (\romannumeral2) SDP solve to generate an extreme ray (\cf Algorithm~\ref{alg:ray:extremeray-decomp-robust}, Line~\ref{alg:ray:line:facial-reduction-3}); (\romannumeral3) pivoted QR decomposition to remove redundant constraints in facial reduction (\cf Algorithm~\ref{alg:ray:remove-dependent-cols}); and (\romannumeral4) alternating projection to improve feasibility (\cf Algorithm~\ref{alg:ray:alternating-projection}).

At iteration $k$, the iterate $X_k$ has numerical rank $r_k \le r_\max$. Computing its eigenvalue decomposition costs $O(N^3)$ time and $O(N^2)$ memory. Solving the reduced SDP by an interior-point method costs $O(\Tipm r_k^6)$ time and $O(r_k^4)$ memory, where $\Tipm$ denotes the number of IPM iterations and is typically moderate~\cite{vandenberghe96siam-review-sdp}. The pivoted QR step takes as input a matrix of size $r_k^2 \times \abs{\calI_{k-1}}$, and therefore costs $O(\min\{r_k^4 \abs{\calI_{k-1}},\ r_k^2 \abs{\calI_{k-1}}^2\})$ time and $O(r_k^2 \abs{\calI_{k-1}})$ memory. By definition, we can upper bound $\abs{\calI_{k-1}}$ by $m$. The complexity of alternating projection depends on the sparsity pattern of $\calA$, so we now specialize to the pseudo-moment cone $\calC = \MomConePs{n+1,2d}$, which is the main focus of this paper.

In this case, $N = \binom{n+d}{d} \sim O(n^d)$ and $m = \frac{1}{2}(N+1)N - \binom{n+2d}{2d} \sim O(N^2)$. Each constraint in $\calA(X)=0$ has the form $X_{\alpha,\beta} - X_{\alpha',\beta'} = 0$, where $\alpha+\beta=\alpha'+\beta'$ and $\alpha,\beta,\alpha',\beta' \in \MultiIndex{n+1}{d}$. For each $\gamma \in \MultiIndex{n+1}{2d}$, define $\calE_\gamma := \{(\alpha,\beta) \in (\MultiIndex{n+1}{d})^{\times 2} \mid \alpha \le \beta,\ \alpha+\beta=\gamma\}$ and let $s_\gamma := \abs{\calE_\gamma}$. Then $\{\calE_\gamma\}_{\gamma \in \MultiIndex{n+1}{2d}}$ partitions the upper-triangular entries of $X$, so $\sum_{\gamma \in \MultiIndex{n+1}{2d}} s_\gamma = \frac{1}{2}N(N+1)$ and $m = \sum_{\gamma \in \MultiIndex{n+1}{2d}} (s_\gamma-1)$. For each $\gamma$, fix an ordering $\calE_\gamma = \{(\alpha_{\gamma,j},\beta_{\gamma,j})\}_{j=1}^{s_\gamma}$ and choose the independent constraints as the chain $X_{\alpha_{\gamma,j},\beta_{\gamma,j}} - X_{\alpha_{\gamma,j-1},\beta_{\gamma,j-1}} = 0$ for $j=2,\dots,s_\gamma$. With this choice, each row of $\calA$ has exactly two nonzero entries, rows from different $\gamma$ have disjoint supports, and $\calA\calA\tran$ is block diagonal:
\begin{align}
    \label{eq:ray:complexity-AAT}
    \calA \calA\tran
    =
    \operatorname{blkdiag}\bigl(T_{s_\gamma-1}\bigr)_{\gamma \in \MultiIndex{n+1}{2d}},
    \qquad
    T_q =
    \begin{bmatrix}
        2 & -1 \\
        -1 & 2 & -1 \\
        & \ddots & \ddots & \ddots \\
        & & -1 & 2
    \end{bmatrix}.
\end{align}
Since each $T_q$ is positive definite and has a bidiagonal Cholesky factor, the sparse Cholesky factorization of $\calA\calA\tran$ costs $O(\sum_{\gamma \in \MultiIndex{n+1}{2d}} (s_\gamma-1)) = O(m) = O(N^2)$ time and memory.

Moreover, the linear-subspace projection in Algorithm~\ref{alg:ray:alternating-projection}, Line~\ref{alg:ray:line:alt-subspace}, is given by $X \mapsto X - \calA^\ast (\calA\calA^\ast)^{-1}\calA(X)$. Under the above chain basis, applying $\calA$, applying $\calA^\ast$, and solving $(\calA\calA^\ast)z=\calA(X)$ with the precomputed sparse Cholesky factor each cost $O(m)=O(N^2)$. Hence one linear-subspace projection costs $O(N^2)$ time and memory. By contrast, the PSD projection still costs $O(N^3)$ time and $O(N^2)$ memory, so one call of Algorithm~\ref{alg:ray:alternating-projection} costs $O(\Talt N^3)$ time and $O(N^2)$ memory. The one-time sparse Cholesky preprocessing cost $O(N^2)$ is absorbed by this bound.

In summary, for $\calC=\MomConePs{n+1,2d}$, the time complexity of one inner-loop iteration in $\textsc{RayDecompFR}$ is bounded by
$
    O (\Tipm r_\max^6 + r_\max^4 N^2 + \Talt N^3 ),
$
and the corresponding memory complexity is bounded by
$
    O ( r_\max^4 + r_\max^2 N^2 ).
$
Since the total number of inner-loop iterations is $O(r_\max)$, the overall time complexity of Algorithm~\ref{alg:ray:extremeray-decomp-robust} is
$
    O ( r_\max(\Tipm r_\max^6 + r_\max^4 N^2 + \Talt N^3) ).
$
In the worst case $r_\max=N$, and if $\Tipm$ and $\Talt$ are treated as constants, this becomes $O(N^7) \sim O(n^{7d})$ time and $O(N^4) \sim O(n^{4d})$ memory.


\section{Numerical Experiments}
\label{sec:exp}

In this section, we evaluate the practical efficiency and numerical robustness of Algorithm~\ref{alg:ray:extremeray-decomp-robust} under finite-precision arithmetic. In \S\ref{sec:exp:atomic}, we assess its performance as a numerical atomic decomposition method for moment matrices. In \S\ref{sec:exp:highrank}, we examine regimes in which the algorithm no longer recovers the original weighted atoms, yet still produces a collection of generally non-unique extreme rays of $\MomConePs{n,2d}$. In particular, the experiments reveal the emergence of high-rank extreme rays, suggesting that Algorithm~\ref{alg:ray:extremeray-decomp-robust} may also serve as a practical sampler for at least a subset of the high-rank extreme rays of $\MomConePs{n,2d}$. All experiments were carried out on a high-performance workstation equipped with a 2.7\,GHz AMD 64-Core sWRX8 processor and 1\,TB of RAM. The implementation and experimental results are available at \url{https://github.com/ComputationalRobotics/momentcone-simregularity}.

\subsection{Algorithm~\ref{alg:ray:extremeray-decomp-robust} as Atomic Decomposition of Moment Matrix}
\label{sec:exp:atomic}

The hyper-parameters in Algorithm~\ref{alg:ray:extremeray-decomp-robust} are fixed throughout the experiments as $\epsrank = 10^{-7}$, $\epsbreak = 10^{-4}$, $\epscol = 10^{-7}$, $\epsalt = 10^{-14}$, and $\Talt = 500$. We use the commercial solver \MOSEK~\cite{aps19ugrm-mosek-sdpsolver} to solve the SDP subproblems. For each pair $(n,d)$, we set $\calC = \MomConePs{n+1,2d}$ and $N = \binom{n+d}{d}$, and sweep the number of atoms $s$ from $2$ to $N$. For each $s$, we generate $100$ random moment matrices $X_i(n,d,s) \in \calC$, $i = 1,\dots,100$, of the form $\sum_{j=1}^s c_{i,j} m_d(z_{i,j}) m_d(z_{i,j})\tran$, where the coefficients $c_{i,j}$ are sampled i.i.d. from the uniform distribution on $[0,1]$, and each entry of $z_{i,j} \in \Real{n}$ is sampled i.i.d. from the uniform distribution on $[-1,1]$. For each input $X_i(n,d,s)$, Algorithm~\ref{alg:ray:extremeray-decomp-robust} returns a decomposition $\{t_k,M_k\}_{k=1}^r$. We regard the recovery as unsuccessful if either $r \neq s$ or at least one output matrix $M_k$ has numerical rank greater than $1$. In this case, we set the relative recovery errors for the weights and atoms to be $e_{w,i}(n,d,s) = 1$ and $e_{z,i}(n,d,s) = 1$, respectively. Otherwise, we recover the weights $\{c_k'\}_{k=1}^s$ and atoms $\{z_k'\}_{k=1}^s$ using Corollary~\ref{cor:ray:recover}. Without loss of generality, we assume that both the ground-truth weights $\{c_k\}_{k=1}^s$ and the recovered weights $\{c_k'\}_{k=1}^s$ are arranged in increasing order, and compare the two decompositions accordingly. The relative recovery errors are then defined by
\begin{align*}
    e_{w,i}(n,d,s) = \frac{1}{s} \sum_{k=1}^s \frac{\abs{c_k^\prime - c_k}}{1 + c_k}, \quad 
    e_{z,i}(n,d,s) = \frac{1}{s} \sum_{k=1}^s \frac{\normtwo{z_k^\prime - z_k}}{1 + \normtwo{z_k}}.
\end{align*}
Finally, we average over the $100$ trials and report $e_w(n,d,s) := \frac{1}{100} \sum_{i=1}^{100} e_{w,i}(n,d,s)$ and $e_z(n,d,s) := \frac{1}{100} \sum_{i=1}^{100} e_{z,i}(n,d,s)$.

\paragraph{Results.}
We consider two choices of the relaxation order \(d\): (\romannumeral1) for \(d=2\), we sweep \(n\) from \(3\) to \(10\); and (\romannumeral2) for \(d=3\), we sweep \(n\) from \(2\) to \(5\). The cases \((d,n)=(2,3)\) and \((3,2)\) correspond to \(\SosCone{4,4}\) and \(\SosCone{3,6}\), respectively, the two smallest cases in which the SOS cone differs from the cone of nonnegative forms~\cite{blekherman12book-semidefinite}. Figure~\ref{fig:exp:error} reports \(e_w\) and \(e_z\) as functions of \(s \in [2,N]\). For each \((d,n)\), we observe a clear phase transition: below an empirical threshold, both recovery errors remain essentially zero, whereas above this threshold they rapidly increase to \(1\). Table~\ref{tab:exp:sbound} summarizes these empirical phase-transition thresholds and compares them with the theoretical upper bound on \(s\) from~\eqref{eq:sim:minface-s}. The empirical thresholds are consistently much larger than the theoretical ones, suggesting that the bound in~\eqref{eq:sim:minface-s} is conservative in practice. Finally, Table~\ref{tab:exp:time} reports the running time for a single sample \(X_i\) with \(s=\binom{n+d}{d}\), illustrating the efficiency of the implementation.

To further illustrate the scalability and robustness of Algorithm~\ref{alg:ray:extremeray-decomp-robust}, we consider a larger-scale instance with $d=2$, $n=20$, and $s=100$. In this setting, the algorithm successfully recovers all $100$ weighted atoms in $7$ h $12$ min.


\begingroup

\newcommand{\FieldFig}[3]{
    \begin{minipage}{0.23\textwidth}
        \centering
        \includegraphics[width=\columnwidth]{figs/recover/k=#1_n=#2/beautiful_error.png}
        (#3) $(d,n) = (#1, #2)$
    \end{minipage}
}

\begin{figure}[htbp]
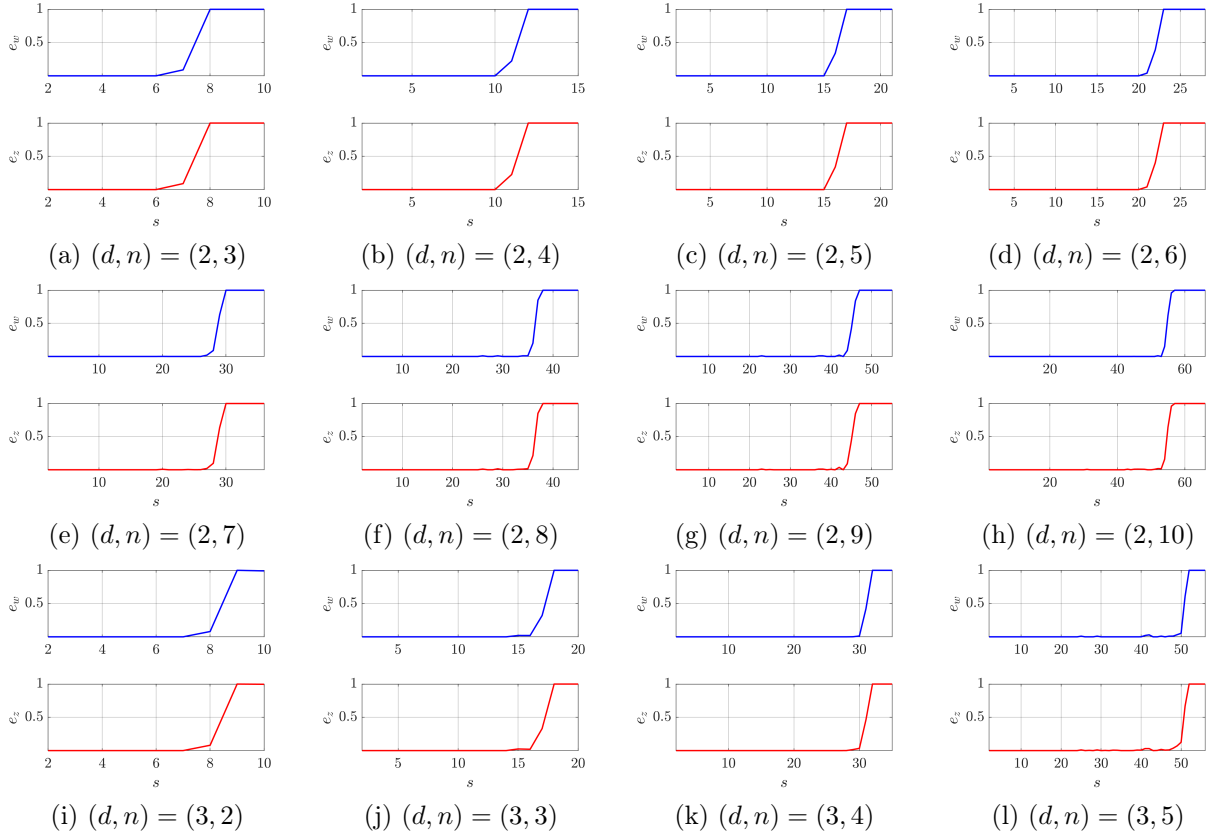

    \centering

    \begin{minipage}{\textwidth}
        \centering
        \begin{tabular}{cccc}
            \FieldFig{2}{3}{a}
            \FieldFig{2}{4}{b}
            \FieldFig{2}{5}{c}
            \FieldFig{2}{6}{d}
        \end{tabular}
    \end{minipage}

    \begin{minipage}{\textwidth}
        \centering
        \begin{tabular}{cccc}
            \FieldFig{2}{7}{e}
            \FieldFig{2}{8}{f}
            \FieldFig{2}{9}{g}
            \FieldFig{2}{10}{h}
        \end{tabular}
    \end{minipage}

    \begin{minipage}{\textwidth}
        \centering
        \begin{tabular}{cccc}
            \FieldFig{3}{2}{i}
            \FieldFig{3}{3}{j}
            \FieldFig{3}{4}{k}
            \FieldFig{3}{5}{l}
        \end{tabular}
    \end{minipage}

    \caption{\label{fig:exp:error} Response curves of $e_w$ and $e_z$ as functions of $s$ for different $d$ and $n$.}
\end{figure}

\endgroup

\begin{table}[htbp]
    \centering
    \begin{tabular}{ccccccccccccc}
        \toprule
        $d$      & 2    & 2    & 2    & 2    & 2    & 2     & 2     & 2 & 3 & 3 & 3 & 3   \\
        \midrule
        $n$      & 3    & 4    & 5    & 6    & 7    & 8     & 9     & 10 & 2 & 3 & 4 & 5    \\
        \midrule
        Theoretical Upper Bound from~\eqref{eq:sim:minface-s}   & 2 & 3 & 5 & 7 & 9 & 12 & 15 & 18 & 2 & 6 & 12 & 20 \\
        \midrule
        Empirical Upper Bound from Experiments   & 6 & 10 & 15 & 20 & 26 & 35 & 43 & 53 & 7 & 16 & 29 & 47 \\
        \bottomrule
    \end{tabular}
    \caption{Upper bounds on $s$ for the uniqueness of the extreme-ray decomposition for different $d$ and $n$. We report both the theoretical upper bounds given by~\eqref{eq:sim:minface-s} and the empirical upper bounds inferred from the phase transitions of $e_w$ and $e_z$ shown in Figure~\ref{fig:exp:error}. \label{tab:exp:sbound}}
\end{table}

\begin{table}[htbp]
    \centering
    \begin{tabular}{ccccccccccccc}
        \toprule
        $d$ & 2 & 2 & 2 & 2 & 2 & 2 & 2 & 2 & 3 & 3 & 3 & 3 \\
        \midrule
        $n$ & 3 & 4 & 5 & 6 & 7 & 8 & 9 & 10 & 2 & 3 & 4 & 5 \\
        \midrule
        $t$ (s) & 0.46 & 0.53 & 1.13 & 3.57 & 6.32 & 16.49 & 46.45 & 82.72 & 0.42 & 1.32 & 6.05 & 34.03 \\
        \bottomrule
    \end{tabular}
    \caption{Elapsed time for extreme ray decomposition with different $d$ and $n$. $s$ is set to $\binom{n+d}{d}$. \label{tab:exp:time}}
\end{table}

\subsection{\titlemath{Algorithm~\ref{alg:ray:extremeray-decomp-robust} as Sampler of High-Rank Extreme Rays of $\MomConePs{n,2d}$}}
\label{sec:exp:highrank}

We now report an intriguing empirical phenomenon: once $s$ crosses the phase-transition threshold, \ie once it enters the regime in which both $e_w$ and $e_z$ are nearly $1$ in Figure~\ref{fig:exp:error}, high-rank extreme rays of $\MomConePs{n+1,2d}$ begin to appear. For example, when $(d,n) = (2,3)$, Figure~\ref{fig:exp:n=3} shows snapshots of the numerical ranks of the extreme rays returned from $100$ random initial matrices $X_i$ for different values of $s \in [3,N]$. When $s \le 6$, all returned extreme rays have rank $1$.  Once $s > 7$, however, rank-$6$ extreme rays appear for almost all random seeds. This observation is consistent with the theoretical results in~\cite{blekherman12jams-nonnegative-polynomial-sos,blekherman12book-semidefinite}, which show that the extreme rays of $\SosCone{4,4}^*$ have rank either $1$ or $6$. The numerical ranks \(3\) and \(5\) observed when \(s=8\) are artifacts of finite-precision rank truncation; they should be interpreted as rank-\(1\) and rank-\(6\) extreme rays, respectively.

A similar pattern is observed for other $(d,n)$ pairs. We refer the reader to our \href{https://github.com/ComputationalRobotics/momentcone-simregularity}{codebase} for animations of how the extreme-ray ranks evolve as $s$ increases. Figure~\ref{fig:exp:n=others} presents snapshots of the numerical ranks of the returned extreme rays when $s=N$. From these plots, we read off the following rank ranges: (a) $(d,n) = (2,4)$: $8,9$; (b) $(d,n) = (2,5)$: $12$--$14$; (c) $(d,n) = (2,6)$: $16$--$19$; (d) $(d,n) = (2,7)$: $21$--$25$; (e) $(d,n) = (2,8)$: $27$--$32$ (with $31$ absent); (f) $(d,n) = (2,9)$: $33$--$40$ (with $38$ and $39$ absent); (g) $(d,n) = (2,10)$: $39$--$48$; (h) $(d,n) = (3,2)$: $7$; (i) $(d,n) = (3,3)$: $14,15$; (j) $(d,n) = (3,4)$: $25$--$28$; and (k) $(d,n) = (3,5)$: $41$--$47$ (with $45$ and $46$ absent). At the same time, our current algorithm does not appear to sample the full spectrum of extreme rays of $\MomConePs{n+1,2d}$. Indeed,~\cite{blekherman15pams-positive-gorenstein-ideals} shows that rank-$6$ extreme rays exist for all $d=2$ and $n \ge 4$, yet no such rays appear in Figure~\ref{fig:exp:n=others}.
Another interesting empirical observation is that, within a single decomposition trajectory, higher-rank extreme rays \emph{always} appear only after lower-rank ones. For example, when $(d,n) = (2,4)$, Algorithm~\ref{alg:ray:extremeray-decomp-robust} typically encounters a sequence of rank-$1$ extreme rays first, followed by rank-$8$ extreme rays, and finally rank-$9$ extreme rays. This consistent ordering suggests that the algorithm may traverse extreme rays in a structured manner, with lower-rank rays appearing earlier in the decomposition and higher-rank rays emerging only at later stages.
Explaining these empirical phenomena would be an interesting direction for future work. 


\begingroup

\newcommand{\FieldFig}[1]{
    \begin{minipage}{0.23\textwidth}
        \centering
        \includegraphics[width=\columnwidth]{figs/recover/k=2_n=3/rank_k=2_n=3_pe=#1.png}
    \end{minipage}
}

\begin{figure}[htbp]
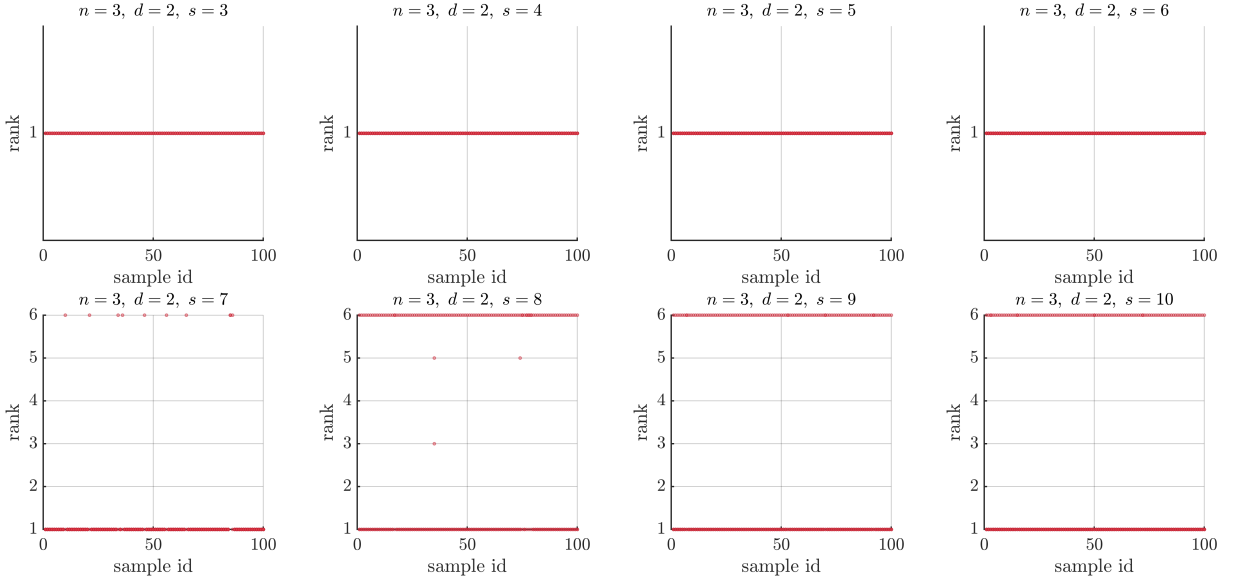

    \centering

    \begin{minipage}{\textwidth}
        \centering
        \begin{tabular}{cccc}
            \FieldFig{3}
            \FieldFig{4}
            \FieldFig{5}
            \FieldFig{6}
        \end{tabular}
    \end{minipage}

    \begin{minipage}{\textwidth}
        \centering
        \begin{tabular}{cccc}
            \FieldFig{7}
            \FieldFig{8}
            \FieldFig{9}
            \FieldFig{10}
        \end{tabular}
    \end{minipage}

    \caption{\label{fig:exp:n=3} Numerical ranks of extreme rays for $d = 2, n = 3$ when $s$ is increased from $3$ to $10$. Each point on $x$-axis represents a random seed.}
\end{figure}

\endgroup

\begingroup

\newcommand{\FieldFig}[4]{
    \begin{minipage}{0.23\textwidth}
        \centering
        \includegraphics[width=\columnwidth]{figs/recover/k=#1_n=#2/rank_k=#1_n=#2_pe=#3.png}
        (#4) $(d,n) = (#1,#2)$
    \end{minipage}
}

\begin{figure}[htbp]
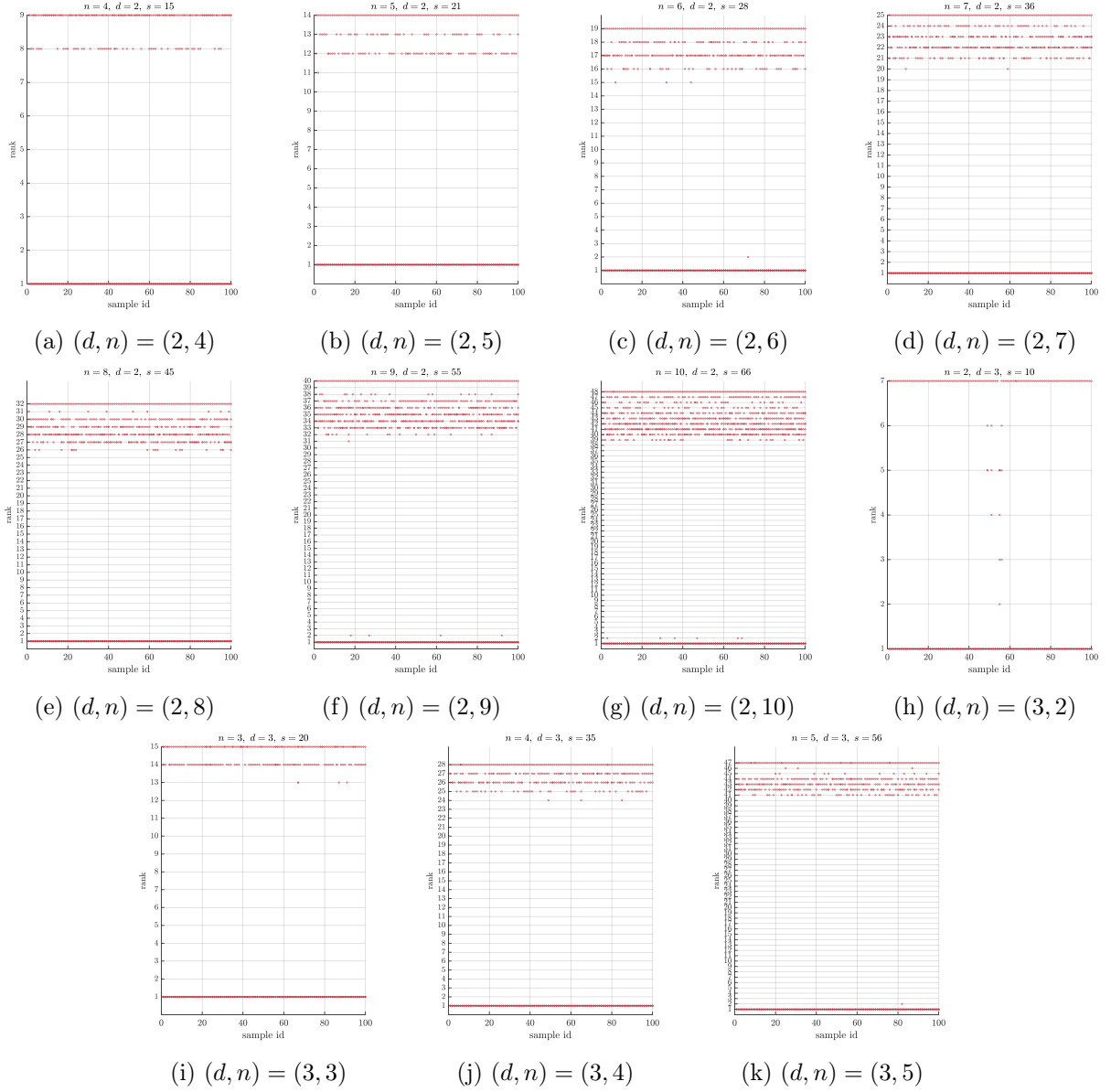

    \centering

    \begin{minipage}{\textwidth}
        \centering
        \begin{tabular}{cccc}
            \FieldFig{2}{4}{15}{a}
            \FieldFig{2}{5}{21}{b}
            \FieldFig{2}{6}{28}{c}
            \FieldFig{2}{7}{36}{d}
        \end{tabular}
    \end{minipage}

    \begin{minipage}{\textwidth}
        \centering
        \begin{tabular}{cccc}
            \FieldFig{2}{8}{45}{e}
            \FieldFig{2}{9}{55}{f}
            \FieldFig{2}{10}{66}{g}
            \FieldFig{3}{2}{10}{h}
        \end{tabular}
    \end{minipage}

    \begin{minipage}{\textwidth}
        \centering
        \begin{tabular}{ccc}
            \FieldFig{3}{3}{20}{i}
            \FieldFig{3}{4}{35}{j}
            \FieldFig{3}{5}{56}{k}
        \end{tabular}
    \end{minipage}

    \caption{\label{fig:exp:n=others} Numerical ranks of extreme rays for different $n$ and $d$, with $s = \binom{n+d}{d}$. Each point on $x$-axis represents a random seed.}
\end{figure}

\endgroup


\section{Conclusion}
\label{sec:conclusion}

We established a new facial-geometric description of the pseudo-moment cone via simplicial regularizability around generically generated moment matrices. Based on this structure, we developed a Carath\'eodory-type extreme-ray decomposition method for spectrahedral cones and showed that, for pseudo-moment cones, it gives an efficient atomic decomposition algorithm recovering generically generated moment matrices with up to \( \Theta(n^d) \) atoms. We further introduced a stabilized implementation and demonstrated strong practical performance in numerical experiments, including the appearance of high-rank extreme rays beyond the proven regime. It would be interesting to extend this framework to broader term-sparsity settings~\cite{wang21siopt-tssos} and non-commutative settings~\cite{klep18siopt-minimizer-extraction-robust}, and to further improve the efficiency of Algorithm~\ref{alg:ray:extremeray-decomp-robust} while clarifying its role as a sampler of high-rank extreme rays of \( \SosCone{n,2d}^* \).

\subsection*{Acknowledgments}

We are grateful to Grigoriy Blekherman and Pablo A.~Parrilo for valuable discussions on high-rank extreme rays of the pseudo-moment cone. We also thank Joseph Kileel and Bobby Shi for valuable discussions on tensor decomposition. Part of this work was conducted during visits by SK and HY to the Polynomial Optimization team at LAAS-CNRS, where Jean-Bernard Lasserre, Didier Henrion, and Victor Magron shared helpful insights.





\bibliographystyle{plain}
\bibliography{../../../references/refs,../../../references/myRefs}

\end{document}